%% file: multiagentopt.tex
\documentclass[11pt,onecolumn]{IEEEtran}

\usepackage{amssymb,amsmath, amsthm}
\usepackage{subfigure}
\usepackage{graphicx}
\usepackage{booktabs}
\usepackage{xcolor}
\usepackage{url}
\usepackage[noadjust]{cite}
\usepackage{hyperref}
\usepackage{movie15}

\usepackage{tikz}
\usetikzlibrary{arrows}
\usetikzlibrary{graphs}
\usetikzlibrary{graphs.standard}
\usetikzlibrary{decorations.markings}
\usetikzlibrary{shapes}

\newcommand{\din}{d^{\text{in}}}
\newcommand{\dout}{d^{\text{out}}}
\newcommand{\dink}{d^{\text{in}, k}}
\newcommand{\doutk}{d^{\text{out}, k}}
\newcommand{\Nin}{N^{\text{in}}}
\newcommand{\Nout}{N^{\text{out}}}
\newcommand{\Nink}{N^{\text{in},k}}
\newcommand{\Noutk}{N^{\text{out},k}}
\def\R{\mathbb{R}}
\newcommand{\1}{{\bf 1}}
\newcommand{\minimize}{\mathop{\operatorname{minimize}}}

\newtheorem{theorem}{Theorem}
\newtheorem{corollary}[theorem]{Corollary}
\newtheorem{proposition}[theorem]{Proposition}
\newtheorem{lemma}[theorem]{Lemma}
\newtheorem{assumption}{Assumption}

\def\ao{}
\def\an{}

\begin{document}

\title{Network Topology and Communication-Computation Tradeoffs in Decentralized Optimization}
\author{Angelia Nedi\'{c}, Alex Olshevsky, and Michael~G.~Rabbat%
\thanks{A.~Nedi\'{c} is with the School of Electrical, Computer, and Energy Engineering, Arizona State University, Tempe, AZ, USA.}%
\thanks{A.~Olshevsky is with the Department of Electrical and Computer Engineering, Boston University, Boston, MA, USA.}%
\thanks{M.G.~Rabbat is with Facebook AI Research, Montr\'{e}al, Canada, and the Department of Electrical and Computer Engineering, McGill University, Montr\'{e}al, Canada.}%
\thanks{Email: angelia.nedich@asu.edu, alexols@bu.edu, michael.rabbat@mcgill.ca}%
\thanks{The work of A.N.\ and A.O.\ was supported by the Office of Naval Research under grant number N000014-16-1-2245. The work of A.O.\ was also supported by NSF under award CMMI-1463262 and AFOSR under award FA-95501510394. The work of M.R.\ was supported by the Natural Sciences and Engineering Research Council of Canada under awards RGPIN-2012-341596 and RGPIN-2017-06266.}
}
\maketitle

\begin{abstract} 
In decentralized optimization, nodes cooperate to minimize an overall objective function that is the sum (or average) of per-node private objective functions. Algorithms interleave local computations with communication among all or a subset of the nodes. Motivated by a variety of applications---decentralized estimation in sensor networks, fitting models to massive data sets, and decentralized control of multi-robot systems, to name a few---significant advances have been made towards the development of robust, practical algorithms with theoretical performance guarantees. This paper presents an overview of recent work in this area. In general, rates of convergence depend not only on the number of nodes involved and the desired level of accuracy, but also on the structure and nature of the network over which nodes communicate (e.g., whether links are directed or undirected, static or time-varying). We survey the state-of-the-art algorithms and their analyses tailored to these different scenarios, highlighting the role of the network topology.
\end{abstract}

\section{Introduction}
\label{sec:intro}

In multi-agent consensus optimization, $n$ agents or \emph{nodes}, as we will refer to them throughout this article, cooperate to solve an optimization problem. A local objective function $f_i : \R^d \rightarrow \R$ is associated with each node $i=1,\dots,n$, and the goal is for all nodes to find and agree on a minimizer of the average objective $f(x) = \frac{1}{n} \sum_{i=1}^n f_i(x)$ in a decentralized way. Each node maintains its own copy $x_i \in \R^d$ of the optimization variable, and node $i$ only has direct access to information about its local objective $f_i$; for example, node $i$ may be able to calculate the gradient $\nabla f_i(x_i)$ of $f_i$ evaluated at $x_i$. Throughout this article we focus on the case where the functions $f_i$ are convex (so $f$ is also convex) and where $f$ has a non-empty set of minimizers so that the problem is well-defined.

Because each node only has access to local information, the nodes must communicate over a network to find a minimizer of $f(x)$. Multi-agent consensus optimization algorithms are iterative, where each iteration typically involves some local computation followed by communication over the network.

\subsection{Architectures for Distributed Optimization}

Gradient descent is a simple, well-studied, and widely-used method for solving minimization problems, and it is one of the first methods one typically studies in a course on numerical optimization~\cite{Bertsekas2003,BoydVandenberghe,NocedalWright}. Gradient descent is a prototypical \emph{first-order} method because it only makes use of gradients $\nabla f(x) \in \R^d$ of a continuously differentiable objective function $f : \R^d \rightarrow \R$ to find a minimizer $x^*$, gradients being the first-order derivatives of $f$. It is useful to discuss how one may implement gradient descent in a distributed manner in order to build intuition for the multi-agent optimization methods on which we focus in this article.

Centralized gradient descent for minimizing the function $f(x)$ starts with an initial value $x^0$ and recursively updates it for $k = 1, 2, \dots,$ by setting
\begin{equation}
x^{k+1} = x^k - \alpha_k \nabla f(x^k), \label{eqn:gradient_descent}
\end{equation}
where $\alpha_1, \alpha_2, \dots,$ is a sequence of positive scalar step-sizes. When $f$ is convex, it has a unique minimum, and it is well-known that, for appropriate choices of the step-sizes $\alpha_k$, the sequence of values $f(x^k)$ converges to this minimum~\cite{Bertsekas2003,BoydVandenberghe,NocedalWright}.

Now, recall the multi-agent setup, where
\begin{equation}
f(x) = \frac{1}{n} \sum_{i=1}^n f_i(x) \label{eqn:ma_objective}
\end{equation}
and where the gradient $\nabla f_i(x)$ can only be evaluated at agent $i$. There are a variety of distributed architectures one may consider in this setting, and we discuss three here: 1) the master-worker architecture, 2) the fully-connected architecture, and 3) a general, connected architecture. They are depicted in Fig.~\ref{fig:architectures} and described next.

\begin{figure*}
\centering
\subfigure[][Master-worker]{\input{master-worker.tex}} \hspace*{1em}
\subfigure[][Fully-connected]{\input{fully-connected.tex}} \hspace*{1em}
\subfigure[][General connected]{\input{general-connected.tex}}
\caption{Three example architectures for distributed optimization. (a) In the master-worker architecture, each agent sends and receives messages from the master node. (b) In a fully-connected architecture, each agent sends and receives messages to every other agent in the network. (c) In a general multi-agent architecture, each agent only communicates with a subset of the other agents in the network. Note, in these figures the positions of each agent aren't meant to reflect geographic locations; rather, the aim is just to depict the communication topology.}
\label{fig:architectures}
\end{figure*}
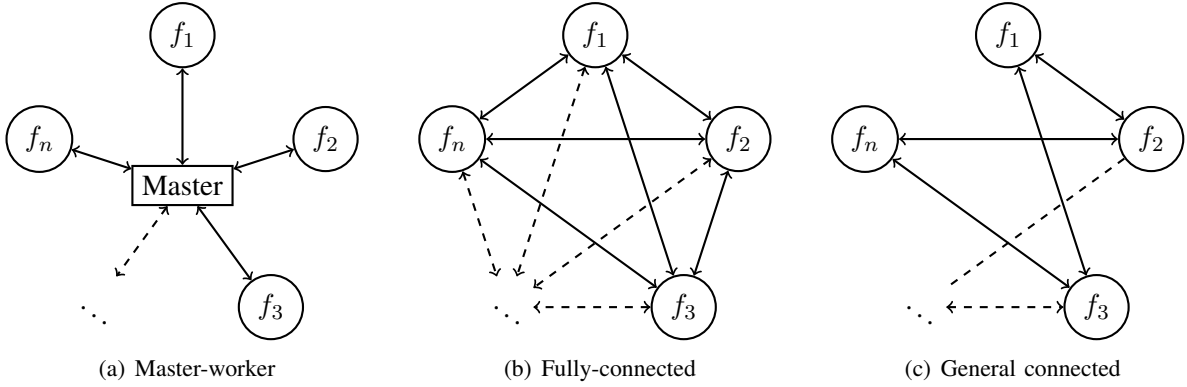

\subsubsection{The master-worker architecture} When $f(x)$ decomposes as in \eqref{eqn:ma_objective}, then the gradient also decomposes, and
\[
\nabla f(x) = \frac{1}{n}\big( \nabla f_1(x) + \nabla f_2(x) + \dots + \nabla f_n(x) \big),
\]
that is, the gradient of the overall objective is the average of the gradients of the local objectives.

In a master-worker architecture, one node acts as the \emph{master} (sometimes also called the fusion center), maintaining the authoritative copy of the optimization variable $x^k$. At each iteration, it sends $x^k$ to every agent, and agent $i$ returns $\nabla f_i(x^k)$ to the master. The master averages the gradients it receives from the agents, and once it has received a gradient from every agent it can perform the gradient descent update \eqref{eqn:gradient_descent}, before proceeding to the next iteration.

The master-worker architecture is useful in that it is relatively simple to implement. However in many applications it may be unattractive or impractical for a variety of reasons. First, as the number of nodes $n$ grows large, the master node may become a communication bottleneck if it has limited communication resources (e.g., if its bandwidth does not grow linearly with the size of the network), and at the same time, scaling the bandwidth of the master with the size of the network may be expensive or impractical. Also, the master node may become a robustness bottleneck, in the sense that if the master node fails then the entire network fails. In addition, in many scenarios it may not be practical to have a single master node that communicates with all agents. For example, if agents are low-power devices communicating via wireless radios, then two devices may only be able to communicate if they are nearby and it may not be practical to have all nodes within the required proximity of the master.

\subsubsection{The fully-connected architecture} A natural first step to address some of the issues of the master-worker architecture is to eliminate the master node, leading to a fully-connected peer-to-peer architecture, where each node communicates directly with all other nodes. In this case, each node $i=1,\dots,n$ maintains a local copy of the optimization variable, $x_i^k \in \R^d$. To mimic centralized gradient descent in a similar way, suppose that the local copies of the optimization variables are initialized to the same value, e.g., $x_i^0 = 0$ for all $i=1,\dots,n$. Then, each node computes its local gradient $\nabla f_i(x_i^0)$ and sends it to every other node in the network. Once a node has received gradients from all other nodes, it can average them, and since $x_i^0$ was initialized to be the same at every node, we have that
\[
\frac{1}{n} \sum_{j=1}^n \nabla f_j(x_j^0) = \nabla f(x_i^0), \quad \text{ for all } i=1,\dots,n.
\]
Thus, using the average of the gradients received from its neighbors, node $i$ can update
\begin{equation}
x_i^1 = x_i^0 - \alpha_0 \left(\frac{1}{n} \sum_{j=1}^n \nabla f_j(x_j^0)\right) \label{eqn:fully-connected-update}
\end{equation}
and $x_i^1$ is \emph{exactly} equivalent to having taken one step of centralized gradient descent. Furthermore, the values $x_i^1$ will be identical at all nodes, and so we can repeat this process recursively to essentially implement centralized gradient descent exactly in a distributed manner.

For the fully-connected architecture just described\footnote{We refer to it as \emph{fully-connected} because every node communicates with every other node at each iteration.}, each node acts like a master in the master-worker architecture, and so the fully-connected architecture suffers from the same issues as the master-worker architecture. Moreover, the communication overhead of having all nodes communicate at every iteration is even worse than the master-worker architecture (it grows quadratically in the number of nodes $n$, whereas the communication overhead was linear in $n$ for the master-worker architecture). Nevertheless, the fully-connected architecture provides a conceptual transition from the master-worker architecture to general connected (but not fully-connected) architectures.

\subsubsection{General multi-agent architectures} Consider a peer-to-peer architecture where node $i$ is only connected to a subset of the other nodes, and not necessarily all of them. Let $N_i \subset \{1,\dots,n\}$ denote the \emph{neighbors} of node $i$: the subset of nodes that sends messages to node $i$. Similar to the fully connected case, suppose that $x_i^0 \in \R^d$ is initialized to the same value at all nodes, and let $x_i^k$ denote the value at node $i$ after $k$ iterations. We can approximately implement gradient descent in a decentralized manner by mimicking the update \eqref{eqn:fully-connected-update}, but where agent $i$ only averages over the gradients it receives from its neighbors, so that
\begin{equation}
x_i^{k+1} = x_i^k - \alpha_k \left( \frac{1}{|N_i|} \sum_{j \in N_i} \nabla f_j(x_j^k) \right), \label{eqn:multi-agent-update}
\end{equation}
where $|N_i|$ is the size of node $i$'s neighborhood. 

This approach given in \eqref{eqn:multi-agent-update} is prototypical of most multi-agent optimization algorithms, in that the update equation can be implemented in the following steps, which are executed in parallel at every node, $i=1,\dots,n$:
\begin{enumerate}
\item Node $i$ locally computes $\nabla f_i(x_i^k)$.
\item Node $i$ transmits its gradient $\nabla f_i(x_i^k)$ and receives gradients $\nabla f_j(x_j^k)$ from its neighbors $j \in N_i$.
\item Node $i$ uses this new information to compute the new value $x_i^{k+1}$, e.g., via equation \eqref{eqn:multi-agent-update}.
\end{enumerate}
Different multi-agent optimization algorithms may differ in terms of what information gets exchanged in the second step, and in the precise way they compute the update in the last step, not necessarily using \eqref{eqn:multi-agent-update}, as well as in the assumptions they make about the local objective functions $f_i$ or the communication topology, captured by the neighborhoods $N_i$. For example: the communication topology may be static or it may vary from iteration to iteration; communications may be undirected (node $i$ receives messages from node $j$ if and only if $j$ also receives messages from $i$) or directed.

The general multi-agent approach to implementing gradient descent, given in \eqref{eqn:multi-agent-update}, also raises a set of issues which did not come up in the other architectures. Since the master-worker and fully-connected architectures exactly implement gradient descent, the well-established convergence theory for gradient descent directly applies to those architectures. However, when nodes update using the rule \eqref{eqn:multi-agent-update}, they no longer exactly implement centralized gradient descent, because they use a search direction
\[
\frac{1}{|N_i|} \sum_{j \in N_i} \nabla f_j(x_j^k) \ne \nabla f(x_i^k)
\]
which is the average of a subset, rather than all, of the gradients at other nodes. Thus, after the first iteration, the local values $x_j^1$ at different nodes are no longer equivalent. Subsequently, at the next iteration, the local gradients being averaged at node $i$ will have been evaluated at different values $x_j^1$. Therefore, there are a few ways in which the values produced by multi-agent gradient descent and deviates from centralized gradient descent. One may hope that, under the right conditions, the values at different nodes will not be too different from each other and that the local search directions will be sufficiently similar to the gradient search direction that the nodes still converge to (and agree on!) a minimizer of $f(x)$. 

Indeed, we will see that we can identify a variety of conditions under which multi-agent optimization algorithms are guaranteed to converge, and we can precisely quantify how the convergence rate differs from that of centralized gradient descent. Most often, this difference depends directly on the communication topology. In many applications of interest, either it is not possible or one does not allow each node to communicate with every other node. The connectivity of the network (i.e., which pairs of nodes may communicate directly with each other) can be represented as a graph with $n$ vertices and with an edge from $j$ to $i$ if node $j$ receives messages from node $i$. We will see that the communication network topology plays a key role in the convergence theory of multi-agent optimization methods in that it may limit the flow of information between distant nodes and thereby hinder convergence.

During the past decade, multi-agent consensus optimization has been the subject of intense interest, motivated by a variety of applications which we discuss in next.

\subsection{Motivating Applications}

The general multi-agent optimization problem described above was originally introduced and studied in the 1980's in the context of parallel and distributed numerical methods \cite{Tsi1986,tsitsiklisPhDThesis,bertsekasParallel}. The surge of interest in multi-agent convex optimization during the past decade has been fueled by a variety of applications where a network of autonomous agents must coordinate to achieve a common objective. We describe three such examples next; for a survey describing additional applications of multi-agent methods for coordination, see \cite{cao2013overview}.

\subsubsection{Decentralized estimation} Consider a wireless sensor network with $n$ nodes where node $i$ has a measurement $\zeta_i$ which is modeled as a random variable with density $p(\zeta_i | x)$ depending on unknown parameters $x$. {\color{black}For example, the network may be deployed to monitor a remote or difficult to reach location, and the estimate of $x$ may be used for scientific observation (e.g., bird migration patterns) or for detecting events (e.g., avalanches) \cite{barrenetxea2008sensorscope}.} In many applications of sensor networks, uncertainty is primarily due to thermal measurement noise introduced at the sensor itself, and so it is reasonable to assume that the observations $\{\zeta_i\}_{i=1}^n$ are conditionally independent given the model parameters $x$. In this case, the maximum likelihood estimate of $x$ can obtained by solving
\[
\minimize_x \quad - \sum_{i=1}^n \log p(\zeta_i | x),
\]
which can be addressed by using multi-agent consensus optimization methods~\cite{rabbat2004distributed,schizas2008consensus2,cattivelli2010diffusion,kar2011convergence,kar2012distributed,kar2014asymptotically} with $f_i(x) = - \log p(\zeta_i | x)$. 

{\color{black}In this example, the data are already being gathered in a decentralized manner at different sensors. When the data dimension is large (e.g., for image or video sensors), it can be more efficient to perform decentralized estimation and simply transmit the estimate of $x$ to the end-user, rather than transmitting the raw data and then performing centralized estimation \cite{rabbat2004distributed}. Similarly, even if the data dimension is not large, if the number $n$ of nodes in the network is large, it may still be more efficient to perform decentralized estimation rather than sending raw data to a fusion center, since the fusion center will become a bottleneck.}

When the nodes communicate over a wireless network, whether or not a given pair of nodes can directly communicate is typically a function of their physical proximity as well as other factors (e.g., fading, shadowing) affecting the wireless channel, which may possibly result in time-varying and directed network connectivity.

\subsubsection{Big data and machine learning} Many methods for supervised learning (e.g., classification or regression) can also be formulated as fitting a model to data. This task may generally be expressed as finding model parameters $x$ by solving
\begin{equation}
\minimize_x \quad \sum_{j=1}^m l_j(x), \label{eq:ml_training}
\end{equation}
where the loss function $l_j(x)$ measures how well the model with parameters $x$ describes the $j$th training instance, and the training data set contains $m$ instances in total. For many popular machine learning models---including linear regression, logistic regression, ridge regression, the LASSO, support vector machines, and their variants---the corresponding loss function is convex~\cite{understandingML}.

When $m$ is large, it may not be possible to store the training data on a single server, or it may be desirable, for other reasons, to partition the training data across multiple nodes (e.g., to speedup training by exploiting parallel computing resources, or because the data is being gathered and/or stored at geographically distant locations). In this case, the training task~\eqref{eq:ml_training} can be solved using multi-agent optimization with local objectives of the form~\cite{ram2009new,tsianos2012consensus,chen2012diffusion,Duchi2012,Jakovetic2011a}
\[
f_i(x) = \sum_{j \in \mathcal{J}_i} l_j(x),
\]
where $\mathcal{J}_i$ is the set of indices of training instances at node $i$.

In this setting, where the nodes are typically servers communicating over a wired network, it may be feasible for every node to send and receive messages from all other nodes. However, it is often still preferable, for a variety of reasons, to run multi-agent algorithms over a network with sparser connectivity. Communicating a message takes time, and reducing the number of edges in the communication graph at any iteration corresponds to reducing the number of messages to be transmitted. This results in iterations that take less time and also that consume less bandwidth.

\subsubsection{Multi-robot systems} Similar to the previous example, multi-agent methods have attracted attention in applications requiring the coordination of multiple robots because they naturally lead to decentralized solutions. One well-studied problem arising in such systems is that of \emph{rendezvous}---collectively deciding on a meeting time and location. When the robots have different battery levels or are otherwise heterogeneous, it may be desirable to design a rendezvous time and place, and corresponding control trajectories, which minimize the energy to be expended collectively by all robots. This can be formulated as a multi-agent optimization problem where the local objective $f_i(x)$ at agent $i$ quantifies the energy to be expended by agent $i$ and $x$ encodes the time and place for rendezvous~\cite{cortes2006robust,cao2013overview}.

When robots communicate over a wireless network, the network connectivity will be dependent on the proximity of nodes as well as other factors affecting channel conditions, similar to in the first example. Moreover, as the robots move, the network connectivity is likely to change. It may be desirable to ensure that a certain minimal level of network connectivity is maintained while the multi-robot system performs its task, and such requirements can be enforced by introducing constraints in the optimization formulation~\cite{zavlanos2013network}.

\subsection{Outline of the rest of the paper}

The purpose of this article is to provide an overview of the main advances in this field, highlighting the state-of-the-art methods and their analyses, and pointing out open questions. During the past decade, a vast literature has amassed on multi-agent optimization and related methods, and we do not attempt to provide an exhaustive review (which, in any case, would not be feasible in the space of one article). Rather, in addition to describing the main advances and results leading to the current state-of-the-art, we also seek to provide an accessible survey of theoretical techniques arising in the analysis of multi-agent optimization methods. 

{\color{black}
As we have already seen, decentralized averaging algorithms---where each node initially holds a number or vector, and the aim is to calculate the average at every node---form a fundamental building block of multi-agent optimization methods. Section~\ref{sec:consensus} reviews decentralized averaging algorithms and their convergence theory in the setting of undirected graphs, where node $i$ receives message from node $j$ if and only if $j$ also receives messages from $i$,. The main results of this section provide conditions under which decentralized averaging algorithms converge asymptotically to the exact average, and they quantify how close the values at each node are to the exact average after a finite number $k$ of iterations. We initially consider the general scenario where the communication topology is time-varying, finding that a sufficient condition for convergence is that the topology be sufficiently well-connected over periodic windows of time. Then, for the particular case where the communication topology is static, we present stronger results illustrating how the rate of convergence depends intimately on properties of the communication topology.

With these results for decentralized averaging in hand, Section~\ref{sec:undirected} presents multi-agent optimization methods and theory in the setting of undirected communication networks. This section reviews convergence theory for the centralized subgradient method, which generalizes gradient descent to handle convex functions which are not necessarily differentiable; such functions arise in a variety of important contemporary applications, such as estimators using $\ell_1$ regularization (e.g., the LASSO) or fitting support vector machines. The main results of this section establish conditions for convergence of decentralized, multi-agent subgradient descent, including quantifying how the rate of convergence depends on the network topology. This section also discusses recent advances which lead to faster convergence in certain settings, such as when the network size is known in advance, or when the objective function is strongly convex.

Section~\ref{sec:directed} then describes how decentralized averaging and multi-agent optimization methods can be extended to run over networks with directed connectivity (i.e., where node $i$ may receive messages from $j$ although $j$ does not receive messages from $i$). The key technique we study, which enables this extension, is the so-called ``push-sum'' approach. We provide a novel, concise analysis of the push-sum method for decentralized averaging, and then we describe how it can be used to obtain decentralized optimization methods.

Section~\ref{sec:extensions} discusses a variety of ways that the basic approaches described in Sections~\ref{sec:undirected} and~\ref{sec:directed} can be extended. For example, in both Sections~\ref{sec:undirected} and~\ref{sec:directed} we limit our attention to methods for unconstrained optimization problems running in a synchronous manner. Sections~\ref{sec:extensions} discusses how to handle constrained optimization problem and how multi-agent optimization methods can be implemented in an asynchronous manner. It also describes other extensions, such as handling stochastic gradient information or online optimization (where the objective function varies over time), and discusses connections to other methods for distributed optimization.

We conclude in Section~\ref{sec:conclusion} and highlight some open problems. 
}

\subsection{Notation}

Before proceeding, we summarize some notation that is used throughout the rest of this paper. A matrix is called \emph{stochastic} if it is nonnegative and the sum of the elements in each row equals one. A matrix is called \emph{doubly stochastic} if, additionally, the sum of the elements in each column equals one. To a stochastic matrix $A \in \R^{n \times n}$, we associate the directed graph $G_A$ with vertex set $\{1, \ldots, n\}$ and edge set $E_A = \{ (i,j) ~|~ a_{ji} > 0 \}$. Note that this graph may contain self-loops. A directed graph is \emph{strongly connected} if there exists a directed path from any initial vertex to every other vertex in the graph. For convenience, we abuse notation slightly by using the matrix $A$ and the graph $G_A$ interchangeably; for example, we say that $A$ is strongly connected. Similarly, we say that the matrix $A$ is \emph{undirected} if $(i,j) \in E_A$ implies $(j,i) \in E_A$. Finally, we denote by $[A]_{\alpha}$ the thresholded matrix obtained from $A$ by setting every element smaller than $\alpha$ to zero.

Given a sequence of stochastic matrices $A^0, A^1, A^2 \ldots$, \ao{for $k > l$,} we denote by $A^{k:l}$ the product of matrices $A^k$ to $A^l$ inclusive, i.e., 
\[ A^{k:l} = A^k A^{k-1} \cdots A^l. \] We say that a matrix sequence is {\em $B$-strongly-connected} if the graph with vertex set $\{1, \ldots, n\}$ and edge set \[ \bigcup_{k = lB}^{(l+1)B-1} E_{A^k} \] is strongly connected for each $l=0,1,2,\ldots$. Intuitively, we partition the iterations $k=1,2,\ldots$ into consecutive blocks of length $B$,  and the sequence is $B$-strongly-connected when the graph obtained by unioning the edges within each block is always strongly connected.
When the graph sequence is undirected, we will simply say {\em $B$-connected}.

The \emph{out-neighbors} of node $i$ at iteration $k$ refers to the set of nodes that can receive messages from it,
\[
\Noutk_i = \{ j \mid a_{ji}^k > 0 \},
\]
and similarly, the \emph{in-neighbors} of $i$ at iteration $k$ are the nodes from which $i$ receives messages,
\[
\Nink_i = \{ j \mid a_{ij}^k > 0 \}.
\]
We assume that $i$ is always an neighbor of itself (i.e., the diagonal entries of $A^k$ are non-zero), \ao{which means that we always have $i \in \Noutk_i  $ and $i \in \Nink_i$}.
When the graph is not time-varying, we simply refer to the out-neighbors $\Nout_i$ and in-neighbors $\Nin_i$. When the graph is undirected, the sets of in-neighbors and out-neighbors are identical, so we will simply refer to the \emph{neighbors} $N_i$ of node $i$, or $N_i^k$ in the time-varying setting. The \emph{out-degree} of node $i$ at iteration $k$ is defined as the cardinality of $\Noutk_i$ and is denoted by $\doutk_i = | \Noutk_i |$. Similarly, $\dink_i$, $\dout_i$, $\din_i$, $d_i^k$, and $d_i$ denote the cardinalities, respectively, of \an{the sets} $\Nink_i$, $\Nout_i$, $\Nin_i$, $N_i^k$, and $N_i$.

\section{\color{black} Decentralized Averaging over Undirected Graphs}
\label{sec:consensus}

This section reviews methods for decentralized averaging that will form a key building block in our subsequent discussion of methods for multi-agent optimization.

\subsection{Preliminaries: Results for averaging} \label{sec:averaging}

We begin by examining the linear consensus process defined as 
\begin{equation} \label{scons} x^{k+1} = A^k x^k, \quad k=0,1,\dots\end{equation} where the matrices $A^k \in \R^{n \times n}$ are stochastic, and the initial vector $x^0 \in \R^n$ is given. Various forms of Eq.~(\ref{scons}) can be implemented in a decentralized multi-agent setting, and these form the backbone of many decentralized algorithms. 

For example, consider a collection of nodes interconnected in a directed graph and suppose node $i$ holds the $i$'th coordinate of the vector $x^k$. Consider the following update rule:  at  step $k$, node $i$ broadcasts the value $x_i^k$ to its out-neighbors, receives values $x_j^k$ from its in-neighbors, and sets $x_{i}^{k+1}$ to be the average of the messages it has received, so that \[ x_i^{k+1} = \frac{1}{\dink_i} \sum_{j \in \Nink_i} x_j^k. \] This is sometimes called the {\em equal neighbor iteration}, and by stacking up the variables $x_i^k$ into the vector $x^k$ it can be written in the form of Eq.~(\ref{scons}) with an appropriate choice for the matrix $A^k$.

\ao{Intuitively, we may think of the equal-neighbor updates in terms of an opinion dynamics process over a network wherein node $i$ repeatedly revises it's opinion vector $x_i^k$ by averaging the opinions of it's neighbors. As we will see later, under some relatively mild conditions this process converges to a state where all opinions are identical, explaining why Eq. (\ref{scons}) is usually referred to as the ``consensus iteration.''}

Over undirected graphs, an alternative popular choice of update rule is to set 
\[ x_i^{k+1} = x_i^k + \epsilon \sum_{j \in N_i^k} \ao{\left( x_j^k -  x_i^k \right)},\] where $\epsilon > 0$ is sufficiently small. Unfortunately, finding an appropriate choice of $\epsilon$ to guarantee convergence of this iteration can be bothersome (especially when the graphs are time-varying), and it generally requires knowing an upper bound on the degrees of nodes in the network. 

Another possibility (when the underlying graphs are undirected) is the so-called {\em Metropolis update}
\begin{equation} \label{met} x_i^{k+1} = x_i^k + \sum_{j \in N_i^k} \frac{1}{\max\{d_i^k, d_j^k\}} \left( x_j^k - x_i^k \right). \end{equation}  The Metropolis update requires node $i$ to broadcast both $x_i^k$ and its degree $d_i^k$ to its neighbors.  Observe that the Metropolis update of Eq.~(\ref{met}) can be written in the form of Eq.~(\ref{scons}) where the matrices $A^k$ are doubly stochastic. 

\ao{A variation on this is the so-called {\em lazy Metropolis update},
\begin{equation} \label{lmet} x_i^{k+1} = x_i^k + \sum_{j \in N_i^k} \frac{1}{2 \max\{d_i^k, d_j^k\}} \left( x_j^k - x_i^k \right), \end{equation} with the key difference being the factor of $2$ in the denominator. It is standard convention within the probability literature that such updates are called ``lazy,'' since they move half as much per iteration. As we will see later, the lazy Metropolis iteration possesses a number of attractive convergence properties.}

\ao{As we have alluded to above,} under certain technical conditions, the iteration of Eq. (\ref{scons}) results in \emph{consensus}, meaning that all of the $x_i^k$ (for $i=1,\dots,n$) approach the same value as $k \rightarrow \infty$. We describe one such condition next. The key properties needed to ensure asymptotic consensus are that the matrices $A^k$ should exhibit sufficient connectivity and aperiodicity (in the long term). In the following, we use the shorthand $G^k$ for $G_{A^k}$, the graph corresponding to the matrix $A^k$. The starting point of our analysis is the following assumption. 

\begin{assumption} 
The sequence of directed graphs $G^0, G^1, G^2, \ldots$ is $B$-strongly-connected. 
Moreover, each graph $G^k$ has a self-loop at every node.  
\label{strongconn}
\end{assumption}

As the next theorem shows, a variation on this assumption suffices to ensure that the update of Eq.~(\ref{scons}) converges to consensus. 

\begin{theorem}[\cite{tsitsiklisPhDThesis,jadbabaie03}][\ao{Consensus Convergence over Time-Varying Graphs}]  \label{basicconsthm}
Suppose the sequence of stochastic matrices $A^0, A^1, A^2, \ldots$ has the property that there exists an $\alpha>0$ such that the sequence of graphs $G_{[A^0]_{\alpha}}, G_{[A^1]_{\alpha}}, G_{[A^2]_{\alpha}}, \ldots$ satisfies Assumption~\ref{strongconn}. Then $x(t)$ converges to a limit in ${\rm span}\{{\bf 1}\}$ and the convergence is geometric\footnote{\ao{A sequence of vectors $z(t)$ converges to the limit $z$ {\em geometrically} if $\|z(t) - z\|_2 \leq c \alpha^t$ for some $c \geq 0$, and $0 < \alpha < 1$.}}. Moreover, if all the matrices $A^k$ are doubly stochastic then for all $i=1, \ldots, n$, \[ \lim_{k \rightarrow \infty} x^k = \frac{\sum_{i=1}^n x_i^0}{n}. \]
\end{theorem}

\smallskip

\ao{On an intuitive level, the theorem works by ensuring two things. First, there needs to be an assumption of repeated connectivity in the system over the long-term, and this is what the strong-connectivity condition does. Furthermore, thresholding the weights at some strictly positive $\alpha$ rules out counterexamples where the weights decay to zero with time. Secondly, the assumption that every node has a self-loop rules out a class of counterexamples where the underlying graph is bipartite and the underlying opinions oscillate\footnote{\ao{Indeed, observe that if we do not require that each node has a self-loop,  the dynamics $x^{k+1} = \left( \begin{array}{cc} 0 & 1 \\ 1 & 0 \end{array} \right) x^k$, \an{started at $x^0\ne0$,} would be a counterexample to Theorem 1.}}. Once these potential counterexamples are ruled out, Theorem \ref{basicconsthm} guarantees convergence.} 

We now turn to the proof of this theorem, which while being reasonably short, still builds on a sequence of preliminary lemmas and definitions which we present first. 
Given a sequence of directed graphs $G^0, G^1, G^2, \ldots, $, we say that node $b$ is reachable from node $a$ in time period $k:l$ if there exists a sequence of directed edges $e^k, e^{k-1}, \ldots, e^{l+1}, e^l$, such that: (i) $e^j$ is present in $G^j$ for all $j=l,\dots,k$, (ii) the origin of $e^l$ is $a$, (iii) the destination of $e^k$ is $b$. Note that this is the same as stating that $[W^{k:l}]_{ba} > 0$ if the matrices $W^k$ are nonnegative with $[W^k]_{ij} > 0$ if and only if $(j,i)$ belongs to $G^k$. We use $N^{k:l}(a)$ to denote the set of nodes reachable from node $a$ in time period $k:l$.

The first lemma discusses the implications of Assumption \ref{strongconn} for products of the matrices $A^k$.

\begin{lemma}[\ao{\cite{jadbabaie03,Nedic09a,tsitsiklisPhDThesis}}] Suppose $A^0, A^1, A^2, \ldots$ is a sequence of nonnegative  matrices with the property that there exists $\alpha > 0$ such that the sequence of graphs $G_{[A^0]_{\alpha}}, G_{[A^1]_{\alpha}}, G_{[A^2]_{\alpha}}, \ldots$ satisfies Assumption \ref{strongconn} . Then for any integer $l$, $A^{(l+n)B-1:lB}$ is a strictly positive matrix. In fact, every entry of $A^{(l+n)B-1:lB}$ is at least  $\alpha^{nB}$. \label{reachable}
\end{lemma} 

\ao{The proof, given next, is a mathematical formalization of the observation that sufficiently long paths exist between any two nodes.}

\begin{proof} 
Consider the set of nodes reachable from node $i$ in time period $k_{\rm start}$ to $k_{\rm finish}$ in the graph sequence $G_{[A^0]_{\alpha}}, G_{[A^1]_{\alpha}}, G_{[A^2]_{\alpha}}, \ldots$, \an{and denote this set by $N^{k_{\rm finish}:k_{\rm start}}(i)$}. Since each of these graphs has a self-loop at every node by Assumption \ref{strongconn}, the reachable set \an{can only be enlarged}, i.e., $$N^{k_{\rm finish}:k_{\rm start}}(i) \subseteq N^{k_{\rm finish}+1:k_{\rm start}}(i) \mbox{ for all } i, k_{\rm start}, k_{\rm finish}.$$ A further immediate consequence of Assumption \ref{strongconn} is that if $N^{mB-1:lB}(i) \neq \{1, \ldots, n\}$, then $N^{(m+1)B-1:lB}(i)$ is strictly larger than $N^{mB-1:lB}(i)$, because during times $(m+1)B-1:mB$ there is an edge in some $[G^i]_{\alpha}$ leading from the set of nodes already reachable from $i$ to those not already reachable from $i$. Putting together these two properties, we obtain that from $lB$ to $(l+n)B-1$ every node is reachable, i.e.,  $$N^{(l+n)B-1:lB}(i) = \{1, \ldots, n\}.$$ But since every non-zero entry of $[A^k]_{\alpha}$ is at least $\alpha$ by construction, this implies that $A^{(l+n)B-1:lB} \geq \alpha^{nB}$, and the lemma is proved.
\end{proof}

Lemma~\ref{reachable} tells us that, over sufficiently long horizons, the products of the matrices $A^k$ have entries bounded away from zero. The next lemma discusses what multiplication by such a matrix does to the \emph{spread} \an{of the values in a vector.} 

\begin{lemma}  \label{contractionlemma}
Suppose $W$ is a stochastic matrix, every entry of which is at least $\beta > 0$. If $v = W u$ then 
\[ \max_{i=1, \ldots, n} v_i - \min_{i=1, \ldots, n} v_i \leq (1- 2\beta) \left( \max_{i=1, \ldots, n} u_i - \min_{i=1, \ldots, n} u_i \right). \]
\end{lemma}
\begin{proof} Without loss of generality, let us assume that the largest entry of $u$ is $u_1$ and the smallest entry of $u$ is $u_n$. Then,  for $l \in \{1, \ldots, n\}$, 
\begin{eqnarray*}
v_l & \leq & (1-\beta) u_1 + \beta u_n, \\
 v_l & \geq & \beta u_1 + (1-\beta) u_n, 
\end{eqnarray*}  so that for any $a,b \in \{1, \ldots, n\}$, we have 
\begin{align*} 
v_a - v_b & \leq (1-\beta) u_1 + \beta u_n - \left( \beta u_1 + (1-\beta) u_n \right) \\ 
& = (1-2 \beta) u_1 - (1 - 2 \beta) u_n. 
\end{align*}
\end{proof}

With these two lemmas in place, we are ready to prove Theorem~\ref{basicconsthm}. \ao{Our strategy is to apply Lemma \ref{contractionlemma} repeatedly to show that the spread of the underlying vectors keeps getting smaller. }

\begin{proof}[Proof of Theorem \ref{basicconsthm}] Since we have assumed that $G_{[A^0]_{\alpha}}, G_{[A^1]_{\alpha}}, G_{[A^2]_{\alpha}}, \ldots$ satisfy Assumption \ref{strongconn}, by Lemma \ref{reachable}, we have that 
\[ \left[ A^{lB:(l+n)B-1} \right]_{ij} \geq \alpha^{nB}  \] 
for all $l=0,1,2,\ldots$, and $i,j = 1, \ldots, n$.
Applying Lemma \ref{contractionlemma} gives that 
\begin{align*}
\max_{i=1, \ldots, n} &x_i^{(l+n)B} - \min_{i=1, \ldots, n} x_i^{(l+n)B} \\
&\leq \left( 1 - 2 \alpha^{nB} \right) \left( \max_{i=1, \ldots, n} x_i^{lB} - \min_{i=1, \ldots, n} x_i^{lB} \right). 
\end{align*}
Applying this recursively, we obtain that $|x_a^k -x_b^k| \rightarrow 0$ for all $a,b \in \{1, \ldots, n\}$. 

To obtain further that every $x_i^k$ converges, it suffices to observe that $x_i^k$ lies in the convex hull of the vectors $x^t$ for $t \leq k$. Finally, since each $A^k$ is doubly stochastic, 
\[ \sum_{j=1,\ldots,n} x_j^{k+1} = \1^T x^{k+1} = \1^T A^k x^k = \1^T x^k = \sum_{j=1, \ldots, n} x_j^k, \] where $\1$ denotes a vector with all entries equal to one, and thus all $x_i^k$ must converge to the initial average.
\end{proof} 

A potential shortcoming of the proof of Theorem~\ref{basicconsthm} is that the convergence time bounds it leads to tend to scale poorly in terms of the number of nodes $n$. 
We can overcome this shortcoming as illustrated in the following propositions. 
These results apply to a much narrower class of scenarios, but they tend to provide more effective bounds when they are applicable. 

The first step is to introduce a precise notion of convergence time. Let $T \left(n,\epsilon, \{A^0, A^2, \ldots, \} \right)$ denote the first time $k$ when 
\[ \frac{\left\| x^{k} - \frac{\sum_{i=1}^n x_i^0}{n} \1 \right\| }{\left\| x^{0} - \frac{\sum_{i=1}^n x_i^0}{n} \1 \right\|} \leq \epsilon. \] In other words, the convergence time is defined as the time until the deviation from the mean shrinks by a factor of $\epsilon$. The convergence time is a function of the desired accuracy $\epsilon$ and of the underlying sequence of matrices. In particular, we emphasize the dependence on the number of nodes, $n$. When the sequence of matrices is clear from context, we will simply write $T(n,\epsilon)$.

\begin{proposition} 
Suppose \[ x^{k+1} = A^k x^k \] where each $A^k$ is a doubly stochastic matrix.  Then 
\[ \left\|  x^{k} - \frac{\sum_{i=1}^n x_i^0}{n} \1 \right\|_2 
\leq \left( \sup_{l=0,1,2,\ldots} \sigma_{2} \left( A^l \right) \right)^k 
\left\| x^{0} - \frac{\sum_{i=1}^n x_i^0}{n} \1 \right\|_2, \]
where $\sigma_2(A^l)$ denotes the second-largest singular value of the matrix $A^l$.  \label{secsing}
\end{proposition}

We skip the proof, which follows quickly from the definition of singular value. 

We adopt the slightly non-standard notation 
\begin{equation} \label{lambdadef} 
\lambda =  \sup_{l \geq 0} \sigma_{2} \left( A^l \right), \end{equation} 
so that the previous proposition can be conveniently restated as 
\begin{equation} \label{lambd}  
\left \| x^{k} - \frac{\sum_{i=1}^n x_i^0}{n} \1 \right\| 
\leq \lambda^k \left\| x^{0} - \frac{\sum_{i=1}^n x_i^0}{n} \1 \right\|. \end{equation} 
Recalling that $\log(1/\lambda) \le 1/(1-\lambda)$, a consequence of this equation is that
\begin{equation} \label{tlambd} 
T\left(n,\epsilon, \{A^0, A^1, \ldots, \} \right) = O \left( \frac{1}{1-\lambda} \ln \frac{1}{\epsilon} \right), \end{equation} so the number $\lambda$ provides an upper bound on the convergence rate of decentralized averaging.

In general, there is no guarantee that $\lambda < 1$, and the equations we have derived may be vacuous. Fortunately, it turns out that for the lazy Metropolis matrices on connected graphs, it is true that $\lambda < 1$, and furthermore, for many families of undirected graphs it is possible to give order-accurate estimates on $\lambda$, which translate into estimates of convergence time. This is captured in the following proposition. Note that all of these bounds should be interpreted as \emph{scaling laws}, explaining how the convergence time increases as the network size $n$ increases, when the graphs all come from the same family.

\begin{proposition}[\ao{Network Scaling for Average Consensus via the Lazy Metropolis Iteration}] \label{diff-graphs} 
If each $A^k$ is the lazy Metropolis matrix on the ...
\begin{enumerate} \item ...path graph, then $T(n,\epsilon) = O\left( n^2 \log(1/\epsilon) \right)$. 
\item ...$2$-dimensional grid, then $T(n,\epsilon) = O \left( n \log n \log (1/\epsilon) \right)$. 
\item ...$2$-dimensional torus, then $T(n,\epsilon) = O\left( n \log(1/\epsilon) \right)$.
\item ...$k$-dimensional torus, then $T(n, \epsilon) = O\left( n^{2/k} \log(1/\epsilon) \right).$ 
\item ...star graph, then $T(n, \epsilon) = O\left(n^2 \log (1/\epsilon) \right)$. 
\item ...two-star graphs\footnote{A two-star graph is composed of two star graphs with a link connecting their centers.}, then $T(n, \epsilon) = O\left(n^2 \log (1/\epsilon) \right)$. 
\item ...complete graph, then $T(n,\epsilon) = O(1)$.
\item ...expander graph, then $T(n,\epsilon) = O( \log (1/\epsilon))$.
\item ...Erd\H{o}s-R\'{e}nyi random graph\footnote{An Erd\H{o}s-R\'{e}nyi random graph with $n$ nodes and parameter $p$ has a symmetric adjacency matrix $A$ whose $\binom{n}{2}$ distinct off-diagonal entries are independent Bernoulli random variables taking the value 1 with probability $p$. In this article we focus on the case where $p = (1+\varepsilon) \log(n)/n$, where $\varepsilon > 0$, for which it is known that the random graph is connected with high probability~\cite{durett}.} then with high probability\footnote{A statement is said to hold ``with high probability'' if the probability of it holding approaches $1$ as $n \rightarrow \infty$. In this context, $n$ is the number of nodes of the underlying graph.} $T(n, \epsilon) = O (\log (1/\epsilon))$. 
\item ...geometric random graph\footnote{A geometric random graph is one where $n$ nodes are placed uniformly and independently in the unit square $[0,1]^2$ and two nodes are connected with an edge if their distance is at most $r_n$. In this article we focus on the case where $r_n^2 = (1 +\varepsilon) \log(n) /n$ for some $\varepsilon > 0$, for which it is known that the random graph is connected with high probability~\cite{PenroseRGGbook}.}, then with high probability $T(n, \epsilon) = O( n \log n \log (1/\epsilon))$.
\item ...any connected undirected graph, then $T(n, \epsilon) = O\left(n^2 \log(1/\epsilon)\right).$
\end{enumerate} 
\end{proposition} 

\begin{proof}[Sketch of the proof] 
The spectral gap $1/(1-\lambda)$ can be bounded as  $O({\cal H})$ where ${\cal H}$ is the largest hitting time of the Markov chain whose probability transition matrix is the lazy Metropolis matrix~(\ao{Lemma 2.13 of} \cite{lincons2}). We thus only need to bound hitting times on the graphs in question, and these are now standard exercises. For example, the fact that the hitting time on the path graph is quadratic in the number of nodes is essentially the main finding of the standard ``gambler's ruin'' exercise \ao{(see e.g., Proposition 2.1 of \cite{lpw})}. \ao{The result for the $2$-d grid follows from putting together Theorem 2.1 and Theorem 6.1 of \cite{commute}}. For corresponding results on $2$-d and $k$-dimensional tori, please see \ao{Theorem 5.5 of} \cite{lpw}; \ao{note that we are treating $k$ as a fixed number and looking at the scaling as a function of the number of nodes $n$.} Hitting times on star, two-star, and complete graphs are elementary exercises. The result for an expander graph is a consequence of Cheeger's inequality; see Theorem~6.2.1 in \cite{durett}. For Erd\H{o}s-R\'{e}nyi graphs the result  follows because such graphs are expanders with high probability; see \ao{the discussion on page 170} of \cite{durett}. For geometric random graphs a bound can be obtained by partitioning the unit square into appropriately-sized regions, thereby reducing to the case of a $2$-d grid; see \ao{Theorem 1.1 of} \cite{avin}. Finally the bound for connected graphs is from \ao{Lemma 2.2 of \cite{lincons1}}.
\end{proof} 

Fig.~\ref{graphs} depicts examples of some of the graphs discussed in Proposition~\ref{diff-graphs}. Clearly the network structure affects the time it takes information to diffuse across the graph. For graphs such as the path or 2-d torus, the dependence on $n$ is intuitively related to the long time it takes information to spread across the network. For other graphs, such as stars, the dependence is due to the central node (i.e., the ``hub'') becoming a bottleneck. For such graphs this dependence is strongly related to the fact that we have focused on the Metropolis scheme for designing the entries of the matrices $A^k$. Because the hub has a much higher degree than the other nodes, the resulting Metropolis updates lead to very small changes and hence slower convergence (i.e., $A^k$ is diagonally dominant); see Eq.~\eqref{met}. In general, for undirected graphs in which neighboring nodes may have very different degrees, it is known that faster rates can be achieved by using linear iterations of the form Eq.~\eqref{scons}, where $A^k$ is optimized for the particular graph topology~\cite{Xiao2004,Boyd2006randomized}. However, unlike using the Metropolis weights---which can be implemented locally by having neighboring nodes exchange their degrees---determining the optimal matrices $A^k$ involves solving a separate network-wide optimization problem; see \cite{Boyd2006randomized} for a decentralized approximation algorithm.

\begin{figure*}
\centering
\subfigure[][]{\input{path-graph.tex}} \hspace{1em}
\subfigure[][]{\input{grid-graph.tex}} \hspace{1em}
\subfigure[][]{\input{star-graph.tex}} \hspace{1em}
\subfigure[][]{\input{two-star-graph.tex}}
\caption{Examples of some graph families mentioned in Prop.~\ref{diff-graphs}. (a) Path graph with $n=5$. (b) $2$-d grid with $n=16$. (c) Star graph with $n=6$. (d) Two-star graph with $n=12$.}
\label{graphs}
\end{figure*}
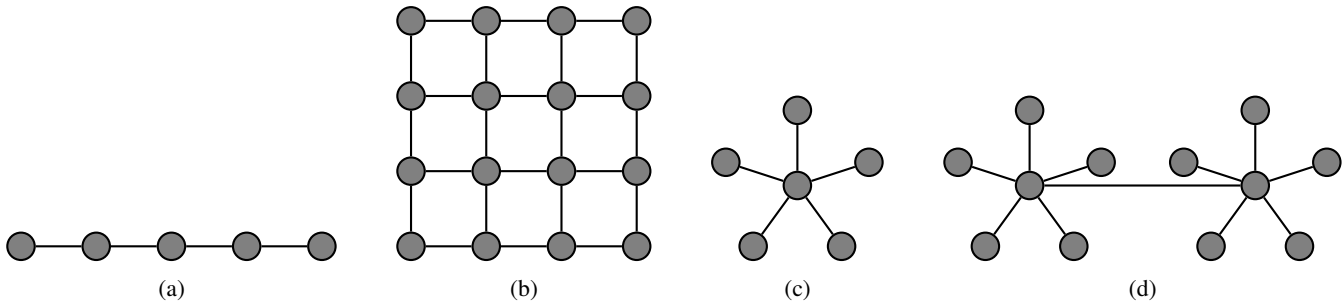

On the other hand, the algorithm is evidently fast on certain graphs. For the complete graph (where every node is directly connected to every other node, this is not surprising---since $A = (1/n) \1 \1^T$, the average is computed exactly at every other node after a single iteration. Expander graphs can be seen as sparse approximations of the complete graph (sparse here is in the sense of having many fewer edges) which approximately preserve the spectrum, and hence the hitting time~\cite{Spielman2017}. In applications where one has the option to design the network, expanders are particularly of practical interest since they allow fast rates of convergence---hence, few iterations---while also having relatively few links---so each iteration requires few transmissions and is thus fast to implement \cite{olfati2007algebraic,kar2008topology}.

\subsection{Worst-case scaling of decentralized averaging} 

One might wonder about the {\em worst-case} complexity of average consensus: how long does it take to get close to the average on any graph? Initial bounds were exponential in the number of nodes~\cite{tsitsiklis1986decentralized,tsitsiklisPhDThesis,bertsekasParallel,jadbabaie03}. However, Proposition~\ref{diff-graphs} tells us that this is at most $O(n^2)$ using a Metropolis matrix. A recent result~\cite{lincons2} shows that if all the nodes know an upper bound $U$ on the total number of nodes which is reasonably accurate, this convergence time can be brought down by an order of magnitude. 
This is a consequence of the following theorem.

\begin{theorem}[\cite{lincons1,lincons2}][\ao{Linear\footnote{\ao{Note that we do not adhere to the common convention of using ``linear convergence'' as a synonym for ``geometric convergece''; rather, ``linear time'' convergence in this paper refers to a convergence time which scales as $O(n)$ in terms of the number of nodes $n$.}} Time Convergence for Consensus}] \label{thm:accelerate}
Suppose each node in an undirected connected graph $G$ implements the update
\begin{align} 
w_i^{k+1}  & =  u_i^k + \frac{1}{2} \sum_{j \in N_i} \frac{u_j^k - u_i^k}{\max(d_i, d_j)}, \nonumber \\
u_i^{k+1} & = w_i^{k+1} + \left(  1 - \frac{2}{9U+1} \right) \left( w_i^{k+1} - w_i^k \right), \label{alm}
\end{align} 
where $U \geq n$ and $u^0 = w^0$. Then
$$\|w^k - \overline{w} \1\|_2^2 \leq 2 \left( 1 - \frac{1}{9U} \right)^{k} \|w^0 - \overline{w} \1\|_2^2,$$ where $\overline{w} = (1/n) \sum_{i=1}^n w_i^0$ is the initial average.  \label{linearconv}
\end{theorem} 

Thus if every node knows the upper bound $U$, the above theorem tells us that the number of iterations until every element of the vector $w^k$ is at most $\epsilon$ away from the initial average $\overline{w}$ is $O(U \ln (1/\epsilon))$. In the event that $U$ is within a constant factor of $n$, (e.g., $n \leq U \leq 10n$) this turns out to be linear in the number of nodes $n$. One situation in which this is possible is if every node precisely knows the number of nodes in the network, in which case they can simply set $U=n$. However, this scheme is also useful in a number of settings where the exact number of nodes in the system is not known (e.g., if nodes fail) as long as approximate knowledge of the total number of nodes is available. 

Intuitively, Eq.~(\ref{alm}) takes a lazy Metropolis update and accelerates it by adding an extrapolation term. Strategies of this form are known as over-relaxation in the numerical analysis literature \cite{varga} and as Nesterov acceleration in the optimization literature \cite{nest}. On a non-technical level, the extrapolation speeds up convergence by reducing the inherent oscillations in the underlying sequence. A key feature, however, is that the degree of extrapolation must be carefully chosen, which is where knowledge of the bound $U$ is required. \an{At present, open questions are} whether any improvement on the quadratic convergence time of Proposition \ref{diff-graphs} is possible without such an additional assumption, and whether a linear convergence time scaling can be obtained for time-varying graphs.

\section{Decentralized optimization over undirected graphs}
\label{sec:undirected}

We now shift our attention from decentralized averaging back to the problem of optimization. We begin by describing the (centralized) subgradient method, which is one of the most basic algorithms used in convex optimization.

\subsection{The subgradient method}

To avoid confusion in the sequel when we discuss decentralized optimization methods, here we consider \an{an iterative} method for minimizing a convex function $h : \R^d \rightarrow \R$. 
A vector $g \in \R^d$ is called a \emph{subgradient} of $h$ at the point $x$ if 
\begin{equation}
h(y) \geq h(x) + g^T (y-x), \quad \mbox{for all } x, y \in \R^d. \label{subgr_def}
\end{equation}
\ao{The subgradient may be viewed as a generalization of the notion of the gradient to non-differentiable (but convex) functions. Indeed,} if the function $h$ is continuously differentiable, then $g=\nabla h(x)$ is the only subgradient at $x$. In general, there are multiple subgradients at points $x$ where the function $h$ is not differentiable. \ao{See Figure~\ref{fig-subgr} for a graphical illustration.}

\begin{figure} \label{fig-subgr}
\begin{center}
    \includegraphics[scale=0.6]{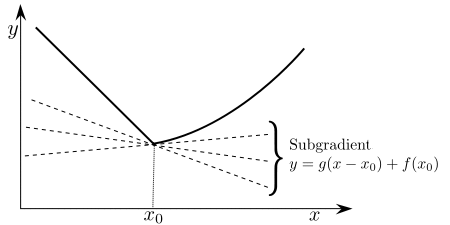}
    \end{center}
\caption{\ao{An illustration of the definition of a subgradient. At the point $x_0$, the function shown is not differentiable. However, there are a number of possible values $g$ such that the tangent line with slope $g$ at $x_0$ is a global understimate of the function, and some of them are shown in the figure. Each such $g$ is a subgradient of the function at $x_0$. The figure is a modified version of an image by Felix Reidel from~\cite{wikipedia}.}}
\end{figure}

The subgradient method\footnote{The earliest work on subgradient methods appears in~\cite{polyak}.} 
for minimizing the function $h$ is defined as the iterate process
\begin{equation} \label{subgr} u^{k+1} = u^k - \alpha^k g^k, \end{equation} 
where $g^k$ is a subgradient of the function $h$ at the point $u^k$. \ao{The quantity $\alpha^k$ is a nonnegative step-size.}

\an{The subgradient method has a somewhat different motivation than the gradient method. It is well known that the gradient is a descent direction at points that are not the global minima. At these points, unlike the gradient, the subgradient gives a direction along which either the function $h(\cdot)$ decreases or the distance toward the set of global minima decreases for small enough stepsizes. 
In general, it is hard to know what is the best step-size choice 
for the convergence of the subgradient method and, as a simple option, 
a diminishing stepsize $\alpha^k$ is commonly used, i.e.,
a stepsize $\alpha^k$ that decreases to zero as $k$ increases. However, to guarantee the convergence toward a global minimum of the function, 
the rate at which the stepsize decreases has to be carefully selected to avoid having the iterates stuck at 
a point that is not a minimizer of the function, 
while controlling the errors that are introduced due to the use of subgradient directions.
This intuition is captured by the following theorem
\footnote{Under weaker assumptions on the stepsize than those of Theorem~\ref{subgradient-theorem}, namely $\alpha^k\to0$ and $\sum_{k=0}^\infty\alpha^k=\infty$, 
one can show that $\liminf_{k\to\infty} h(x_k)=h^*$, 
where $h^*=\inf_{x\in\mathbb{R}^d} h(x)$, see~\cite{shor,polyakbook,Bertsekas2003}.}.}

\begin{theorem}[\ao{Convergence and Convergence Time for the Subgradient Method}]  \label{subgradient-theorem}
Let ${\cal U}^*$ be the set of minimizers of the function $h: \R^d \rightarrow \R$. 
Assume that (i) $h$ is convex, (ii) ${\cal U}^*$ is nonempty, (iii) 
$\|g\| \leq L$ for all subgradients $g$ of the function 
$h(\cdot)$.
\begin{enumerate} \item If the nonnegative step-size sequence $\alpha^k$ 
is ``summable but not square summable,'' i.e., 
\[\sum_{k=0}^{\infty} \alpha^k = +\infty \quad \text{and} \quad \sum_{k=0}^{+\infty} \left[ \alpha^k \right]^2 < \infty. \] 
Then, \an{the iterate sequence $\{u^k\}$ converges to some minimizer $u^*\in U^*$.}

\item If \ao{the subgradient method} is run for $T$ steps with the (constant) choice of stepsize $\alpha^k = 1/\sqrt{T}$ for $k=0, \ldots, T-1$, then  
\[ h \left( \frac{\sum_{k=0}^{T-1} u^k}{T} \right) - h^* \leq \frac{\|u^0 - u^*\|^2 + L^2}{2 \sqrt{T}}, \]
\an{where $h^*$ is the minimal value of the function, i.e., $h^*=h(u^*)$ for any $u^*\in U^*$.}
\end{enumerate}
\end{theorem}
\begin{proof}
\noindent
\an{(1) A proof can be found in Lemma~7 of~\cite{AAdirected}. 
\noindent
(2)
The distance $\|u^k-u^*\|^2_2$, for an arbitrary $u^*\in U^*$ is used to measure the progress of the basic subgradient method. From the definition of the method, for the constant stepsize it follows that
\[\|u^{k+1}-u^*\|^2_2 =\|u^k-u^*\|^2_2- 2\alpha(g^k)^T(u^k-u^*) + \alpha^2 \|g^k\|^2.\]
Then, using the subgradient defining inequality in Eq.~\ref{subgr_def} and the assumption that the subgradient norms are bounded by $L$, we obtain
\[\|u^{k+1}-u^*\|^2_2 =\|u^k-u^*\|^2_2- 2\alpha\left( h(u^k) -h(u^*)\right) + \alpha^2 L^2.\]
By summing these inequalities over $k=0,1,\ldots T$, re-arranging the terms,  and dividing by $2\alpha T$,
one can see that
\[\frac{1}{T}\sum_{k=0}^{T-1}h(u^k) -h(u^*)\le \frac{\|u^k-u^*\|^2_2}{2\alpha T} +\frac{\alpha L^2}{2}.\]
The result follows by using the convexity of $h(\cdot)$ which yields 
\[h\left(\frac{\sum_{k=0}^{T-1}u^k}{T}\right)\le \frac{1}{T}\sum_{k=0}^{T-1}h(u^k),\]
and by letting $\alpha=\frac{1}{\sqrt{T}}.$}
\end{proof}
\an{For the diminishing step in part (1), since the iterates $u^k$ converge to some minimizer $u^*$, so does any weighted average of the iterates (with positive weights). In particular, it follows that
\[\lim_{t \rightarrow \infty}\frac{\sum_{k=0}^{t} \alpha^k u^k}{\sum_{k=0}^t \alpha^k}=u^*.\]
Furthermore, it is a fact that any convex function whose domain is the entire space of the decision variables is continuous at every point. Thus, by continuity of $h(\cdot)$, it follows that 
\[ \lim_{t \rightarrow \infty} h \left( \frac{\sum_{k=0}^{t} \alpha^k u^k}{\sum_{k=0}^t \alpha^k} \right) = h(u^*). \] 
In the case of a fixed stepsize, part (2) provides an error bound in terms of the function values.}
\ao{On a conceptual level, the main takeaway is that the subgradient method produces $O \left(1/\sqrt{T} \right)$ convergence to the optimal function value in terms of the number of iterations $T$.}

\an{Similar to gradients, for the convex functions defined over the entire space, the subgradients are ``linear'' in the sense that 
a subgradient of the sum of two convex functions can be obtained as a sum of two subgradients (one for each function). Formally,
if $g_1$ is a subgradient of a function $f_1$ at $x$ and $g_2$ is a subgradient of $f_2$ at $x$, then $g_1 + g_2$ is a subgradient of $f_1 + f_2$ at $x$. (This follows directly from the subgradient definition in Eq.~\eqref{subgr_def}.)}

\subsection{Decentralizing the subgradient method}

We now return to the problem of decentralized optimization. To recap, we have $n$ nodes, interconnected in a time-varying network capturing which pairs of nodes can exchange messages. For now, assume that these networks are all undirected. (This is relaxed in Section \ref{sec:directed}, which considers directed graphs.) Node $i$ knows the convex function $f_i: \R^m \rightarrow\R$ and the nodes would like to minimize the function 
\begin{equation} 
f(x) =  \frac{1}{n} \sum_{i=1}^n f_i(x) \label{optprob} 
\end{equation}
in a decentralized way.  If all the functions $f_1(x), \ldots, f_n(x)$ were available at a single location, we could directly apply the subgradient method to their average $f(x)$:
\[ u^{k+1} = u^k - \alpha^k \frac{1}{n} \sum_{i=1}^n \overline{g}_i^k, \] where $\overline{g}_i^k$ is a subgradient of the function $f_i(\cdot)$ at $u^k$. Unfortunately, this is not a decentralized method under our assumptions, since only node $i$ knows the function $f_i(\cdot)$, and thus only node $i$ can compute a subgradient of $f_i(\cdot)$. 

A {\em decentralized subgradient method} solves this problem by interpolating between the subgradient method and an average consensus scheme from Section \ref{sec:consensus}. In this scheme, node $i$ maintains the variable $x_i^k$ which is updated as 
\begin{equation} \label{distsub} 
x_i^{k+1} = \sum_{j \in N_i^k} a_{ij}^k x_j^k - \alpha^k g_i^k, \end{equation} 
where $g_i^k$ is the subgradient of $f_i(\cdot)$ at $x_i^k$.  \ao{Here the coefficients $[a_{ij}]$ come from any of the average consensus methods we discussed in Section \ref{sec:consensus}.} \ao{We will refer to this update as the {\em decentralized subgradient method}. }
Note that this update is decentralized in the sense that node $i$ only requires local information to execute it. 
In the case when $f: \R \rightarrow \R$, the quantities $x_j^k$ are scalars and we can write this as
\begin{equation} 
x^{k+1} = A^k x^k - \alpha^k g^k,  \label{eqn:decentralized_subgradient_vectorized}
\end{equation}
where the vector $x^k \in \R^n$ stacks up the $x_i^k$ and $g^k \in \R^n$ stacks up the $g_i^k$. The weights $a_{ij}^k$ should be chosen by each node in a decentralized way. Later within this section, we will assume that the matrices $A^k$ are doubly stochastic; perhaps the easiest way to achieve this is to use the \ao{lazy} Metropolis iteration of Eq. (\ref{lmet}). 

Intuitively, \ao{the decentralized subgradient method of} Eq. (\ref{distsub}) pulls the value $x_i^k$ at each node in two directions: on the one hand towards the minimizer (via the subgradient term) and on the other hand towards neighboring nodes (via the averaging term). \ao{Eq. (\ref{distsub}) can be thought of as reconciling these pulls; note that the strength of the consensus pull does not change, but the strength of the subgradient pull is controlled by the stepsize, and this stepsize  $\alpha^k$ will be \ao{later} chosen to decay to zero, so that in the limit the consensus term will prevail.} However, if the rate at which the stepsize decays to zero is slow enough, 
then under appropriate conditions consensus will be achieved 
\ao{not on some arbitrary point, but rather} on a global minimizer of $f(\cdot)$.

We now turn to the analysis of \ao{the decentralized subgradient method}. For simplicity, we make the assumption that all the functions $f_i(\cdot)$ are from $\R$ to $\R$; this simplifies the presentation but otherwise has no effect on the results. The same analysis extends in a straightforward manner to functions $f_i : \R^d \rightarrow \R$ with $d > 1$ but requires more cumbersome notation.\footnote{\ao{If $x_i^k$ are vectors, one can still stack the per-node vectors into a network-wide vector $x^k$, but then in \eqref{eqn:decentralized_subgradient_vectorized} the matrix $A^k$ must be changed to $A^k \otimes I_d$ where $\otimes$ is the Kronecker product and $I_d$ is the $d \times d$ identity matrix. For such details, we refer the interested reader to~\cite{Nedic09a,Ram2012,Nedic2011}, which do not assume that $x_i^k$ are scalars.}}

\begin{theorem}[\ao{\cite{Nedic09a,Ram2012}}][\ao{Convergence and Convergence Time for the Decentralized Subgradient Method}]
Let ${\cal X}^*$ denote the set of minimizers of the function $f$. We assume that: 
(i) each $f_i$ is convex; 
(ii) ${\cal X}^*$ is nonempty; 
(iii) each function $f_i$ has the property that its subgradients at any point are bounded by the constant $L$;
(iv) the matrices $A^k$ are doubly stochastic and there exists some $\alpha > 0$ such that the graph sequence 
$[G_{A^0}]_{\alpha}, [G_{A^1}]_{\alpha}, [G_{A^2}]_{\alpha}, \ldots$ satisfies Assumption \ref{strongconn}; and
(v) the initial  values $x_i^0$ are the same across all nodes\footnote{\an{This assumption is not necessary for the results stated here, but we use it to simplify the exposition. When this assumption is violated, the bound in part (ii) has an additional term depending on the spread of the initial values. This term decays exponentially on the order of $\lambda^k$.}} (e.g., $x_i^0=0$). Then:
\begin{enumerate} \item  If the \an{positive} step-size sequence $\alpha^k$ is ``summable but not square summable,'' i.e., 
\[ \sum_{k=0}^{\infty} \alpha^k = +\infty \quad \text{and} \quad \sum_{k=0}^{+\infty} \left[ \alpha^k \right]^2 < \infty, \] 
then\footnote{In fact we can show a stronger result that, 
as $k\to\infty$, the iterate sequences 
$\{x_i^k\}$ converge to a common minimizer $x^*\in X^*$,  for all $i$. However, the proof is more involved; see~\cite{Ram2012}.} 
for any $x^* \in {\cal X}^*$, we have that for all 
$i=1, \ldots, n$, 
\[\lim_{t\to\infty}  f \left( \frac{\sum_{l=0}^t \alpha^l x_i^l }{\sum_{l=0}^t \alpha^l} \right)f(x^*). \]
\item 
If we run the protocol for $T$ steps with (constant) step-size $\alpha^k = 1/\sqrt{T}$, and with the notation 
$y^k = (1/n) \sum_{i=1}^n x_i^k$, then we have that for all $i=1, \ldots, n$,
\begin{equation} \label{optconvbound}  f \left( \frac{\sum_{l=0}^{T-1} y^k}{T} \right) - f(x^*) \leq \frac{(y^0 - x^*)^2 + L^2}{2 \sqrt{T}} + \frac{L^2}{\sqrt{T}(1-\lambda)}, \end{equation}
\end{enumerate} where $\lambda$ is defined by Eq. (\ref{lambdadef}).  \label{mainoptthm}
\end{theorem} 

We remark that the quantity $(\sum_{l=0}^{T-1} y^l)/T$ on which the suboptimality bound is proved can be computed via an average consensus protocol after the protocol is finished if node $i$ keeps track of $(\sum_{l=0}^{T-1} x_i^l)/T$. 

Comparing \an{part (2)} of Theorems~\ref{subgradient-theorem} and~\ref{mainoptthm}, and ignoring the similar terms involving the initial conditions, we see that the convergence bound gets multiplied by $1/(1-\lambda)$. This term may be thought of as measuring the ``price of decentralization'' resulting from having knowledge of the objective function decentralized throughout the network rather than available entirely at one place. 

We can use Proposition \ref{diff-graphs} to translate this into concrete convergence times on various families of graphs, as the next result shows. For $\epsilon > 0$, let us define the \emph{$\epsilon$-convergence time} to be the first time when 
\[ f \left( \frac{\sum_{l=0}^{T-1} y^l}{T} \right) - f(x^*) \leq \epsilon. \]  Naturally, the convergence time will depend on $\epsilon$ and on the underlying sequence of matrices/graphs.

\begin{corollary}[\ao{Network Scaling for the Decentralized Subgradient Method}] \label{scalings}
Suppose all the hypotheses of Theorem \ref{mainoptthm} are satisfied, and suppose further that the weights $a_{ij}^k$ are the lazy Metropolis weights defined in Eq. (\ref{lmet}). Then the convergence time can be upper bounded as 
\[ O \left( \frac{  \max \left( (y^0 - x^*)^4, L^4 P_n^2  \right)}{\epsilon^2} \right), \] where if the graphs $G^i$ are
\begin{enumerate} \item ...path graphs, then $P_n = O\left( n^2  \right)$;
\item ...$2$-dimensional grid, then $P_n = O\left( n \log n  \right)$;
\item ...$2$-dimensional torus, then $P_n = O\left( n  \right)$;
\item ...$k$-dimensional torus, then $P_n = O\left( n^{2/k}  \right)$;
\item ...complete graphs, then $P_n = O( 1)$;
\item ...expander graphs, then $P_n = O( 1)$;
\item ...star graphs, then $P_n = O\left(n^2 \right)$;
\item ...two-star graphs, then $P_n = O\left(n^2  \right)$;
\item ...Erd\H{o}s-R\'{e}nyi random graphs, then  $P_n  = O(1)$;
\item ...geometric random graphs, then $P_n = O( n \log n )$;
\item ...any connected undirected graph, then $P_n = O\left( n^2 \right)$.
\end{enumerate} 
\end{corollary} 

These bounds follow immediately by putting together the upper bounds on $1/(1-\lambda)$ discussed in Proposition \ref{diff-graphs} with Eq. (\ref{optconvbound}). 

We remark that it is possible to decrease the scaling from $O(L^4 P_n^2)$ to $O(L^2 P_n)$ in the above corollary if the constant $L$,  the type of the underlying graph (e.g., star graph, path graph), and the number of nodes $n$ is known to all nodes. Indeed, this can be achieved by setting the stepsize $\alpha^k = \beta/\sqrt{T}$ and using a hand-optimized $\beta$ (which will depend on $L$, $n$, as well as the kind of underlying graphs). We omit the details but this is very similar to the optimization done in~\cite{Duchi2012}.

We now turn to the proof of Theorem \ref{mainoptthm}. We will need two preliminary lemmas covering some background in optimization. \ao{The first lemma discusses how the bound on the norms of the subgradients translate into Lipschitz continuity of the underlying function.}

\begin{lemma} \label{p-bound} 
Suppose $h: \R^d \rightarrow \R$ is a convex function such that $h(\cdot)$ has subgradients $g_x, g_y$ at the points $x,y$, respectively, satisfying $\|g_x\|_2 \leq L$ and  $\|g_y\|_2 \leq L$.
Then \[ |h(y) - h(x) | \leq L \|y-x\|_2 \]
\end{lemma}

\begin{proof} On the one hand, we have by definition of subgradient
\[ h(y) \geq h(x) + g_x^T (y-x),  \] so that, \an{by the Cauchy-Schwarz inequality,}
\begin{equation} \label{part1} h(y) - h(x) \geq - L \|y-x\|_2. \end{equation} 
On the other hand 
\[ h(x) \geq h(y) + g_y^T (x-y), \] so that 
\[ h(x) - h(y) \geq - L \|x -y\|_2, \] which we rearrange as 
\begin{equation} \label{part2} h(y) - h(x) \leq L \|y-x\|_2. \end{equation} 
Together Eq. (\ref{part1}) and Eq. (\ref{part2}) imply the lemma.
\end{proof} 

Our overall proof strategy is to view the \ao{decentralized subgradient method} as a kind of perturbed consensus process. To that end, the next lemma extends our previous analysis of the consensus process to deal with perturbations. 

\begin{lemma} \label{pcons}  
Suppose \[ x^{k+1} = A^{k}x^k + \Delta^{k}, \] 
where $A^k$ are doubly stochastic matrices satisfying Assumption \ref{strongconn} and $\Delta^k \in \R^n$ are perturbation vectors.
\begin{enumerate}
\item If $\sup_k \|\Delta^k\|_2 \leq L'$ then 
\[ \left\|x^k - \frac{\1^T x^k}{n} \1 \right\|_2 
\leq \lambda^k \left\|x^0 - \frac{\1^T x^0}{n} \1 \right\|_2 + \frac{L'}{1-\lambda}, \] 
where $\lambda$ is from Eq. (\ref{lambdadef}). 
 \item If $\Delta^{k} \rightarrow 0$, then 
 \an{$x^{k} - \frac{\1^T x^k}{n}\1 \rightarrow 0$.}
\end{enumerate}
\end{lemma} 

\begin{proof} 
For convenience, let us introduce the notation 
$$y^k = \frac{\1^T x^k}{n}, \quad e^k = x^k  - y^k \1, \quad m^k = \frac{1^T \Delta^k}{n}.$$ Since 
\[ y^{k+1}  = y^k   +m^k, \] 
we have that
\[ e^{k+1} = A^k e^k + \Delta^k - m^k \1, \] 
and therefore
\begin{align*}
e^k &= A^{k-1:0} e^0 + A^{k-1:1} (\Delta^0 - m^0 \1) \\
 &\quad + \cdots + A^{k-1} (\Delta_{k-2} - m^{k-2} \1) + (\Delta_{k-1} - m^{k-1} \1). 
\end{align*}
Now using the fact that the vectors $e^0$ and $\Delta^i - m^i \1$ have mean zero, by Proposition \ref{secsing} we have
\begin{equation} \label{ebound} 
\|e^k\|_2 \leq \lambda^{k} \|e^0\|_2 + \sum_{j=0}^{k-1} \lambda^{k-1-j} \|\Delta^{j}\|_2. \end{equation} 
This equation immediately implies the first claim of the lemma. 

Now consider the second claim.  We  define 
\begin{eqnarray*} L_{\rm first-half}^k & = & \sup_{0 \leq j < k/2} \|\Delta^j\|_2  \\ 
L_{\rm second-half}^k & = & \sup_{k/2 \leq j \leq k} \|\Delta^j\|_2  .
\end{eqnarray*} 
Since $\Delta^k \rightarrow 0$, we have 
\begin{equation} \label{bounds} 
\sup_{k\ge2} L_{\rm first-half}^k < \infty, \qquad
\lim_{k\to\infty} L_{\rm second-half}^k = 0. \end{equation}
Now Eq. (\ref{ebound}) implies
\[ \|e^k\|_2 \leq \lambda^{k-1} \|e^0\|_2 + \lambda^{k/2} \frac{L_{\rm first-half}^k}{1-\lambda}  
+ \frac{L_{\rm second-half}^k}{1-\lambda} . \] 
Combining this with Eq. (\ref{bounds}), we have that $\|e^k\|_2 \rightarrow 0$ and 
this proves the second claim of the lemma. 
\end{proof}

\ao{With these lemmas in place, we now turn to the proof of Theorem \ref{mainoptthm}. Our approach will be to view the decentralized subgradient method as a perturbation of a subgradient-like process followed by {\em averaging} of the entries of the vector $x^k$. Provided that the step-size $\alpha^k$ decays to zero at the appropriate rate, we will argue that (i) the vector $x^k$ is not too far from its average and (ii) this average  makes continual progress towards} \an{a minimizer of the function $f(\cdot)$}.

\begin{proof} [Proof of Theorem \ref{mainoptthm}]
Recall that we are assuming, for simplicity,  that the functions $f_i$ are from $\R$ to $\R$. 
As before, let $y^k$ be the average of the entries of the vector $x^k \in \R^n$, i.e., 
\[ y^{k} = \frac{\1^T x^k}{n}. \] 
Since the matrices $A^k$ are doubly stochastic, $\1^T A^k = \1^T$ so that 
\[ y^{k+1} = y^k - \alpha^k \frac{\1^T g^k}{n}. \] 
Now for any $x^* \in {\cal X}^*$, we have  
\begin{equation} \label{basiciterate}  
(y^{k+1} - x^*)^2 \leq (y^k - x^*)^2 + \left[ \alpha^k \right]^2 L^2 - 2 \alpha^k \frac{\sum_{i=1}^n g_i^k}{n} (y^k - x^*). \end{equation} Furthermore, for each $i=1, \ldots, n$,
\begin{eqnarray*} g_i^k (y^k - x^*) & = & g_i^k (x_i^k - x^* + y^k - x_i^k) \\ 
& = & g_i^k (x_i^k - x^*) + g_i^k (y^k - x_i^k ) \\ 
& \geq & f_i(x_i^k) - f_i(x^*) - L \left|y^k - x_i^k \right| \\
& \geq & f_i(y^k) - f_i(x^*) - 2 L \left|y^k- x_i^k \right|, 
\end{eqnarray*} where the first inequality uses a rearrangement of the definition of the subgradient and the last inequality uses  Lemma \ref{p-bound}. Plugging this into Eq. (\ref{basiciterate}), we obtain 
\begin{align*}
(y^{k+1} - x^*)^2 &\leq (y^k - x^*)^2 + \left[ \alpha^k \right]^2  L^2 - 2 \alpha^k ( f(y^k) - f(x^*)) \\
&\quad + 2 L \alpha^k \frac{1}{n} \sum_{i=1}^n \left|y^k - x_i^k \right|,  
\end{align*} 
or 
\begin{align*} 
2 \alpha^k ( f(y^k) - f(x^*)) &\leq  (y^k - x^*)^2  - (y^{k+1} - x^*)^2 \\
&\quad + \left[ \alpha^k \right]^2  L^2  + 2 L \alpha^k \frac{1}{n} \sum_{i=1}^n \left|y^k - x_i^k \right|. 
\end{align*}
We can sum this up to obtain
\begin{align*}\label{xxxx}  
2 \sum_{l=0}^t \alpha^l &\left( f(y^l) - f(x^*) \right) \\
&\leq (y^0 - x^*)^2  - (y^{t+1} - x^*)^2 \\
&\quad + L^2  \sum_{l=0}^t \left[ \alpha^l \right]^2  +  2 L  \sum_{l=0}^t \alpha^l  \frac{1}{n} \sum_{i=1}^n \left|y^l - x_i^l \right|, 
\end{align*}
which in turn  implies 
\begin{align*} 
f&\left(\frac{\sum_{l=0}^t \alpha_l y^l }{\sum_{l=0}^t \alpha_l} \right) - f(x^*) \\
&\leq \frac{ (y^0 - x^*)^2 + L^2 \sum_{l=0}^t \left[ \alpha^l \right]^2  +  2 L  \sum_{l=0}^t \alpha^l (1/n) \sum_{i=1}^n \left|y^l - x_i^l \right|}{2 \sum_{l=0}^t \alpha^l}  \\
&= \frac{ (y^0 - x^*)^2 + L^2 \sum_{l=0}^t \left[ \alpha^l \right]^2 }{2 \sum_{l=0}^t \alpha^l} \\
&\quad + \frac{  2 L  \sum_{l=0}^t \alpha^l (1/n) \sum_{i=1}^n \left|y^l - x_i^l \right|}{2 \sum_{l=0}^t \alpha^l}.
\end{align*} 
We now turn to the first claim of the theorem statement. The first term on the right-hand side goes to zero because its numerator is bounded while its denominator is unbounded (due to the assumption that the step-size is summable but not square summable). For the second term, we view $-\alpha^k g^k$ as the perturbation $\Delta^k$ in Lemma~\ref{scons} to obtain that $x^l - y^l \1 \rightarrow 0$, and in particular $x_i^l - y_i^l \rightarrow 0$ for each $i$. It follows that the Ces\`{a}ro sum (which is exactly the second term on the right-hand side) must go to zero as well. 
We have thus shown that 
\[ f \left( \frac{\sum_{l=0}^t \alpha^l y^l }{\sum_{l=0}^t \alpha^l} \right) - f(x^*) \rightarrow 0. \] 
Putting this together with Lemma~\ref{pcons}, which implies that $x_i^l - y^l \rightarrow 0$ for all $i$, 
we complete the proof of the first claim. 

For the second claim, using (i) Lemma \ref{pcons}, (ii) the assumption that the initial conditions are the same across all nodes, and (iii) the inequalities $\|z\|_1 \leq \sqrt{n} \|z\|_2$ for vectors $z \in \R^n$, we have the bound 
\begin{align} 
f&\left(\frac{\sum_{l=0}^{T-1} y^l}{T} \right) - f(x^*) \nonumber \\
& \leq \frac{(y^0 -x^*)^2 + L^2 + 2 L \sum_{l=0}^{T-1} 1/\sqrt{T} (1/n) \left( \frac{L \sqrt{n} \sqrt{n}}{\sqrt{T}(1-\lambda)} \right)  }{2 \sqrt{T}} \nonumber \\
& = \frac{(y^0 - x^*)^2 + L^2}{2 \sqrt{T}} + \frac{L^2}{\sqrt{T}(1-\lambda)}, \label{finitebound}
\end{align} 
which completes the proof of the second claim. 
\end{proof} 

\subsection{Improved scaling with the number of nodes}  
The results of Corollary \ref{scalings} improve upon those reported in~\cite{Nedic09a, Ram2010, Duchi2012}. 
A natural question is whether it is possible to further improve the scalings even further. In particular, one might wonder how the worst-case convergence time of decentralized optimization scales with the number of nodes in the network. 
In general, this question is open. Partial progress was made in~\cite{lincons1, lincons2}, where, under the assumption that all nodes know an order-accurate bound on the total number of nodes in the network,  it was shown that we can use the linear time convergence of average consensus described in Theorem \ref{linearconv} to obtain a corresponding convergence time for decentralized optimization when the underlying graph is fixed and undirected. Specifically, \cite{lincons1, lincons2} consider the following update rule 
\begin{eqnarray} y_i^{k+1} 
& = & x_i^k +  \frac{1}{2} \sum_{j \in N_i} \frac{x_j^k - x_i^j}{\max(d_i, d_j)} - \beta g_i^k \nonumber \\ 
z_i^{k+1} & = & y_i^k - \beta g_i^k  \label{optaccel} \\ 
x_i^{k+1} & = & y_i^{k+1} + \left(  1 - \frac{2}{9U+1} \right) \left( y_i^{k+1} - z_i^{k+1} \right), \nonumber
\end{eqnarray}  where $g_i^{\ao{k}}$ is a subgradient of the function $f_i(\cdot)$ at the point $y_i^{\ao{k}}$. As in Section \ref{sec:consensus}, here the number $U$ is an upper bound on the number of nodes known to each individual node, and it is assumed that $U$ is within a constant factor of the true number of nodes,  i.e., $n \leq U \leq cn$ for some constant $c$ (not depending on $n,U$ or any other problem parameters).

By relying on Theorem~\ref{linearconv}, it is shown in \cite{lincons1, lincons2} that the corresponding time until this scheme (followed by a round of averaging) is $\epsilon$ close to consensus on a minimizer of $(1/n) \sum_{i=1}^n f_i(\cdot)$ is $O( n \log n + n/\epsilon^2)$. It is an open question at present whether a similar convergence time can be achieved over time-varying graphs or without knowledge of the upper bound $U$.

\subsection{The strongly convex case} 
The error decrease of $1/\sqrt{T}$ with the number of iterations $T$ is, in general, the best possible rate for dimension-independent convex optimization \cite{nemyud}. Under the stronger assumption that the underlying functions $f_i(\cdot)$ are strongly convex with Lipschitz-continuous gradients, gradient descent will converge geometrically. Until recently, however, there were no corresponding decentralized protocols with a geometric rate. 

Significant progress on this issue was first made in \cite{extra}, where, over fixed undirected graphs, 
the following scheme was proposed:
\begin{equation} \label{extra} 
x^{k+2} = (I + W) x^{k+1} - \widetilde{W}x^k - \alpha \left[ \nabla f (x^{k+1}) - \nabla f(x^k) \right], \end{equation} with the initialization \[ x^1 = W x^0 - \alpha \nabla f(x^0). \] Here, for simplicity, we continue with the assumption that the functions $f_i(\cdot)$ are from $\R$ to $\R$. The matrices $W$ and $\widetilde{W}$ are two different, appropriately chosen, symmetric stochastic matrices compatible with the underlying graph; for example, 
$W$ might be taken to be the Metropolis matrix and $\widetilde{W} = (I+W)/2$. 
It was shown in \cite{extra} that this scheme, called EXTRA, drives all nodes to the global optimal at a geometric rate under natural technical assumptions. 

\ao{It is not immediately obvious how to extend the EXTRA update to handle time-varying directed graphs; the original proof in~\cite{extra} only covered static, undirected graphs. Progress on this question was made 
in~\cite{diging} which, in addition to providing a geometrically convergent method in the time-varying and directed cases, also provides a new intuitive interpretation of EXTRA.} Indeed, \cite{diging} observes that the scheme
\begin{eqnarray} x^{k+1} & = & W^k x^k - \alpha y^k \nonumber \\ 
y^{k+1} & = & W^k y^k + \nabla f(x^{k+1}) - \nabla f(x^k) \label{diging}
\end{eqnarray} 
is a special case of the EXTRA update of Eq.~(\ref{extra}). Here, the initialization $x^0$ can be arbitrary, while $y^0 = \nabla f(x^0)$. The matrices $W^k$ are doubly stochastic. Moreover, Eq.~(\ref{diging}) has a natural interpretation. In particular, the second line of Eq.~(\ref{diging}) is a {\em tracking recursion:} $y^k$ tracks the time-varying gradient average $\1^T \nabla f(x^k)/n$. Indeed, observe that, by the double stochasticity of $W^k$, we have that $$\1^T y^k/n = \1^T \nabla f(x^k)/n.$$ In other words, the vector $y^k$ has the same average as the average gradient. Moreover, 
\an{it can be seen} that if $x^k \rightarrow \widehat{x}$, then $y^k \rightarrow \nabla f(\widehat{x})$; this is due to the ``consensus effect'' of repeated multiplications by $W^k$. Such recursions for tracking were studied in~\cite{ZhuM2012}.

While the second line of Eq.~(\ref{diging}) tracks the average gradient, the first line of Eq. (\ref{diging}) performs a \ao{decentralized} gradient step as if $y^k$ was the {\em exact} gradient direction. The method can be naturally analyzed using methods for approximate gradient descent. It was shown in~\cite{diging} that this method converges to the global optimizers geometrically under the same assumptions as EXTRA, even when the graphs are time-varying; further, the complexity of reaching an $\epsilon$ neighborhood of the optimal solution is polynomial in $n$.

We conclude by remarking that there is quite a bit of related work in the literature. Indeed, the idea to use a two-layered scheme as in Eq.~(\ref{diging}) originates 
from~\cite{Xu2015, XuThesis, Lorenzo2015, Lorenzo2016icassp,Lorenzo2016}. Furthermore, improved analysis of convergence rates over an undirected graph is available in~\cite{lina1,lina2}.

\section{Averaging and Optimization Over Directed graphs}
\label{sec:directed}

\subsection{Decentralized averaging over directed graphs}

We have seen in Section \ref{sec:consensus} that over time-varying {\em undirected} graphs, the lazy Metropolis update results in consensus on the initial average. In this section, we ask whether this is possible over a sequence of directed graphs. 

By way of motivation, we remark that many applications of decentralized optimization involve directed graphs. For example, in wireless networks the communication radius of a node is a function of its broadcasting power; if nodes do not all transmit at the same power level, communications will naturally be directed. Any decentralized optimization protocol meant to work in wireless networks must be prepared to deal with unidirectional communications. 

Unfortunately, it turns out that there is no direct analogue of the lazy Metropolis method for average consensus over directed graphs. In fact, if we consider deterministic protocols where, at each step, nodes broadcast information to their neighbors and then update their states based on the messages they have received, then it can be proven that no such protocol can result in average consensus; see  \cite{julien-john}. \ao{The main obstacle is that the consensus iterations we have considered up to now (e.g., in Section \ref{sec:averaging}) relied on doubly stochastic matrices in their updates, which cannot be done over graphs that are time-varying and directed.} We thus need to make an additional assumption to solve the average consensus problem over directed graphs. 

A standard assumption in the field is that {\em every node always knows its out-degree}. In other words, whenever a node broadcasts a message it knows how many other nodes are within listening range. In practice, this can be accomplished in practice via a two-level scheme, wherein nodes broadcast hello-messages at an identical and high power level, while the remainder of the messages are transmitted at lower power levels. The initial exchange of hello-messages provides estimates of distance to neighboring nodes, allowing each node to see how many listeners it has as a function of its transmission power. Alternatively, the out-degrees can be estimated in a decentralized manner using linear iterations~\cite{Charalambous2016decentralized}, assuming that the underlying communication topology is strongly connected.

Under this assumption, it turns out that average consensus is indeed possible and may be accomplished via the following iteration, 
\begin{eqnarray}
 x_i^{k+1} =  \sum_{j \in \Nink_i} \frac{x_j^k}{\dout_j},\qquad
 y_i^{k+1} =  \sum_{j \in \Nink_i} \frac{y_j^k}{\dout_j}, \label{pushsum}
\end{eqnarray} initialized at an arbitrary $x^0$ and $y^0 = \1$.
\an{This is known as the {\em Push-Sum} iteration; it was introduced in \cite{kempe03}, where its correctness was shown for a fully-connected graph (allowing only pairwise communications), while
it was extended to arbitary strongly connected graphs in~\cite{benezit}. In~\cite{dominguez-energy} the push-sum was applied to address distributed energy resources over a static directed graph, with a more recent extensions including imperfect communications such as those with delays in~\cite{hadjicostis2014average} and with packet drops~\cite{hadjicostis2016robust}.}

\ao{On an intuitive level, the update of the variables $x_i^k$ does not lead to consensus because of the lack of doubly stochasticity. Instead, at each time $k$, each $x_i^k$ is some linear combination of $x_j^k$ where $j$ runs over a large enough neighborhood of $i$. The main idea of Push-Sum is that an identical iteration started at the all-ones vector (i.e., the update for $y_i^k$) allows the algorithm to estimate the {\em weights} of that linear combination. Once these weights are known, average consensus can be achieved via rescaling. Indeed, we will show later how a decentralized algorithm can use both $x_i^k$ and $y_i^k$ to achieve average consensus.}

The name Push-Sum derives from the nature of the decentralized implementation of Eq.~(\ref{pushsum}). Observe that Eq.~\eqref{pushsum} may be implemented with one-directional communication. Specifically, every node $i$ transmits (or broadcasts) the values $x_i^k / \dout_i$ and $y_i^k / \dout_i$ to its out-neighbors. After these transmissions, each node has the information it needs to perform the update~\eqref{pushsum}, which involves \emph{summing} the \emph{pushed} values. In contrast, the algorithms for undirected graphs described in the previous sections required that each node $i$ send a message to all of its neighbors \emph{and} receive a message from each neighbor. Protocols of this sort are known as ``push-pull'' in the decentralized computing literature because the transmission of a message from node $i$ to node $j$ (the ``push'') implies that $i$ also expects to receive a message from $j$ (the ``pull'').

Our next theorem, which is the main result of this subsection, tells us that Push-Sum works. 
For this result, we define matrices $A^k$ as follows:
\begin{equation}\label{eq:defA}
a_{ij}^k=\left\{\begin{array}{cc} 
\frac{1}{\dout_j} & \hbox{if $j \in \Nink_i$},\cr
0 & \hbox{else.}\end{array}\right.\end{equation}

\begin{theorem}[\ao{\cite{kempe03,benezit, dominguez}}][\ao{Convergence of Push-Sum}]  \label{pstheorem}
Suppose the sequence of graphs $G_{A^0}, G_{A^1}, G_{A^2}, \ldots$ 
satisfies Assumption~\ref{strongconn}. Then for each $i=1, \ldots, n$,
\[ \lim_{k \rightarrow \infty}  \frac{x_i^k}{y_i^k} = \frac{\sum_{j=1}^n x_j^0}{n}. \]
\end{theorem}

It is somewhat remarkable that the convergence to the average happens for the ratios $x_i^k/y_i^k$. Adopting the notation $a./b$ for the element-wise ratio of two vectors $a$ and $b$, the above theorem may be restated as 
\[ \lim_{k \rightarrow \infty} x^k./y^k = \left( \frac{\sum_{j=1}^n x_j^0}{n} \right) \1. \]

We now turn to the proof of the theorem. Using the matrices $A^k$ as defined in Eq.~\eqref{eq:defA}, 
the iterations in Eq. (\ref{pushsum}) may be written as 
\begin{eqnarray*} x^{k+1} =  A^k x^k, \qquad
y^{k+1}  =  A^k y^k.
\end{eqnarray*} 
Observe that $A^k$ is column stochastic by design, i.e., 
\[ \1^T A^k = \1^T. \]  As a consequence of this, the sums of $x^k$ and $y^k$ are preserved, i.e., 
\begin{eqnarray} 
\sum_{i=1}^n x_i^k  =  \sum_{i=1}^n x_i^0,\qquad
\sum_{i=1}^n y_i^k  =  \sum_{i=1}^n y_i^0 = n.  \label{yupper}
\end{eqnarray}

For our proof, we will need to use the fact that the vector $y^k$ remains strictly positive and bounded away from zero in each entry. This is shown in the following lemma.

\begin{lemma} For all $i,k$, $y_i^k \geq 1/(n^{2nB})$. \label{ybound}
\end{lemma}

\begin{proof}   
By Assumption \ref{strongconn}, every node has a self-loop, so we have that  
\begin{equation} \label{degr} y_i^{k+1} \geq \frac{1}{n} y_i^k \mbox{ for all } i,k,\end{equation} 
and consequently the lemma is true for $k=1, \ldots, nB-1$. Let $lnB$ be the largest multiple of $nB$ which is at most $k$. 
If $k > nB-1$ then
\[ y^k = A^{k:nlB}  A^{nlB-1:0} \1.  \]
By Lemma \ref{reachable}, the matrix $A^{nlB-1:0}$ 
is the transpose product of stochastic matrices satisfying Assumption \ref{strongconn}, and 
consequently each of its entries is at least $\alpha^{nB}$ (where $\alpha = 1/n$) by Lemma \ref{reachable}. 
Thus 
\[ \left[ A^{nlB-1:0} \1 \right]_i \geq \left( \frac{1}{n} \right)^{nB} \mbox{  for all  } i=1, \ldots, n. \] Applying Eq. (\ref{degr}) to the last $k-nlB$ steps now proves the lemma. 
\end{proof}

With this lemma in mind, we can give a proof of Theorem \ref{pstheorem} that is essentially a quick reduction to the result already obtained in Theorem \ref{basicconsthm}.

\begin{proof}[Proof of Theorem \ref{pstheorem}] 
Let us introduce the notation $z_i^k = x_i^k/y_i^k$. Then $x_i^k = z_i^k y_i^k$, and therefore, we can rewrite the Push-Sum update as
\[ z_i^{k+1} y_i^{k+1} = \sum_{j=1}^n [W^k]_{ij} z_j^k y_j^k, \]
or 
\[ z_i^{k+1} = \sum_{j=1}^m [W^k]_{ij} (y_i^{k+1})^{-1} z_j^k y_j^k, \] 
where the last step used the fact that $y_i^k \neq 0$, which follows from Lemma \ref{ybound}. Therefore, defining 
\[ P^k = \left( {\rm diag}(y^{k+1}) \right)^{-1} W^k {\rm diag}(y^k), \] 
we have that
\[ z^{k+1} = P^{k} z^k. \] 
Moreover, $P^k$ is stochastic since
\begin{align*} 
P^k \1 &= \left( {\rm diag}(y^{k+1}) \right)^{-1} W^k {\rm diag}(y^k) \1 \\
&= \left( {\rm diag}(y^{k+1}) \right)^{-1} W^k y^k \\
&= \left( {\rm diag}(y^{k+1}) \right)^{-1} y^{k+1} \\
&= \1. 
\end{align*}
We have thus written the Push-Sum update as an ``ordinary'' consensus update after a change of coordinates. However, to apply Theorem \ref{basicconsthm} about the convergence of the basic consensus process, we need to lower bound the entries of $P^k$, which we proceed to do next. 

Indeed, as a consequence of the definition of $P^k$, if we choose $\alpha$ to be some fixed number such that 
$$\alpha \leq \left( \max_i y_i^{k+1} \right)^{-1}  ~~ \min_{(i,j) ~|~ [W^k]_{ij} > 0} [W^k]_{ij}   ~~ \left( \min_i y_i^k\ \right)$$ 
always holds, then the sequence of graphs $G_{[P^0]_{\alpha}}, G_{[P^1]_{\alpha}}, G_{[P^2]_{\alpha}}, \ldots$ will satisfy Assumption \ref{strongconn}.  To find an $\alpha$ that satisfies this condition, we make use of the fact that $1/(n^{2nB}) \leq y_i^k \leq n$, which is a consequence of Eq. (\ref{yupper}) and  Lemma \ref{ybound}. It follows that the choice $\alpha = (1/n) \cdot (1/n) \cdot 1/(n^{2nB}) $ suffices. Thus, we can apply Theorem \ref{basicconsthm} and obtain that $z^k$ converges to a multiple of the all-ones vector. 

It remains to show that the final limit point is the initial average. Let $z_{\infty} $ be the ultimate limit of each $z_i^k$. Then for each $k=0,1,2,\ldots$,
\begin{eqnarray*} z_{\infty} & = & z_{\infty} \frac{\sum_{i=1}^n y_i^k}{\sum_{i=1}^n y_i^k} \\
& = & \frac{\sum_{i=1}^n z_i^k y_i^k}{n}  + \frac{\sum_{i=1}^n (z_{\infty} - z_i^k) y_i^k}{n} \\ 
& = & \frac{\sum_{i=1}^n x_i^k}{n}  + \frac{\sum_{i=1}^n (z_{\infty} - z_i^k) y_i^k}{n},
\end{eqnarray*}  so 
\[ \lim_{k\to\infty} \left(z_{\infty} -  \frac{\sum_{i=1}^n x_i^k}{n}  \right)
= \lim_{k\to\infty} \frac{\sum_{i=1}^n (z_{\infty} - z_i^k) y_i^k}{n} = 0, \] where the last equality used that each $y_i^k$ is a positive number upper bounded by $n$. Finally, appealing to the first relation in Eq. (\ref{yupper}) we complete the proof of the theorem. 
\end{proof}

\subsection{Push-Sum based subgradient method}
Suppose now that every agent $i$ has a (scalar) convex objective function $f_i(\cdot)$, and the system objective 
is to minimize $f(x)=(1/n)\sum_{i=1}^n f_i(x)$. We  next 
describe a decentralized subgradient method for determining a minimizer of $f$ using the Push-Sum algorithm.
Every node $i$ maintains scalar variables $x_i^k, y_i^k,w_i^k$,
and updates them according to the following rules:
for all $k\ge0$ and all $i=1,\ldots,n$,
\begin{eqnarray}\label{eq:minmet} 
w_i^{k+1} & = & \sum_{j \in \Nink_i} \frac{x_j^k}{\doutk_j},\cr
&&\hbox{}\cr
y_i^{k+1} & = & \sum_{j \in \Nink_i} \frac{y_j^k}{\doutk_j}, \cr
&&\hbox{}\cr
z_{i}^{k+1} & = & \frac{w_{i}^{k+1}} {y_{i}^{k+1}},\cr
&&\hbox{}\cr
x_i^{k+1} &=& w_i^{k+1} - \alpha^{k+1} g_i^{k+1},  
\end{eqnarray} 
where $g_i^{k+1}$ is a subgradient of the function
$f_i(z)$ at $z=z_i^{k+1}$. The method is initiated with an arbitrary vector
$x_i^0\in\R$ at node $i$, and with $y_i^0=1$ for all $i$. 
The Push-Sum updates steer the vectors $z_{i}^{k+1}$ toward each other in order to converge to a common point, while the subgradients in the updates of $x_i^{k+1}$ drive this common point 
to lie in the set of minimizers of the objective function $f$.
 
In the next theorem, we establish the convergence properties of the subgradient method
of Eq.~\eqref{eq:minmet}.
\begin{theorem}[\ao{\cite{AAdirected}}][\ao{Convergence of the Push-Sum Subgradient Method}]\label{mainthm} 
Let ${\cal X}^*$ be the set of minimizers of the function $f$. 
Assume that: (i) each $f_i$ is convex; 
(ii) ${\cal X}^*$ is nonempty; 
(iii) each function $f_i$ has the property that its subgradients at any point are bounded by a constant $L$; and
(iv) the graph sequence $G_{A^0}, G_{A^1}, G_{A^2}, \ldots$ satisfies Assumption~\ref{strongconn}.
\begin{enumerate} 
\item[1)] 
If the stepsizes $\alpha^1, \alpha^2, ...$ are positive, non-increasing, and satisfy the conditions 
\[
\sum_{k=1}^{\infty} \alpha^k=\infty \quad \text{and}
\quad \sum_{k=1}^{\infty} [\alpha^k]^2 < \infty, 
\]
then the decentralized subgradient method of Eq.~\eqref{eq:minmet} converges asymptotically:
\[ \lim_{k \rightarrow \infty} z_i^k = x^* \qquad\mbox{ for all $i$ and for some $x^* \in {\cal X}^*$}. \]
\item[2)] 
If
$\alpha^k = 1/\sqrt{k}$ for $k\ge1$ and every node
$i$ maintains the variable $\widetilde z_i^k \in \R$ initialized at $k=0$ with any
$\widetilde z_i^0\in\R$ and updated by 
\[ \widetilde z_i^{k+1} 
= \frac{\alpha^{k+1} z_i^{k+1} + S^k \widetilde z_i^k}{S^{k+1}}
\quad\hbox{for $k\ge0$},  \] 
where $S^0=0$ and $S^k = \sum_{s=0}^{k -1} \alpha^{s+1}$ for $k\ge1$,  
then for all $k \geq 1$, $i=1, \ldots, n$, and any $x^* \in X^*$,
\begin{align*}  
f&\left(\widetilde z_i^{k+1} \right) - f^* \\
&\leq \frac{n}{2}\frac{|\bar x(0) - x^*|}{\sqrt{k+1}} + \frac{L^2 \left( 1 + \ln (k+1) \right)}{2 n\sqrt{k+1} } \\
&\quad + \frac{ 24 L \sum_{j=1}^n |x_j^0| }
 {   \delta (1-\lambda) \sqrt{k+1} }+ \frac{24 L^2 \left( 1+ \ln k \right) }{\delta (1-\lambda) \sqrt{k+1} }, 
 \end{align*}
where $f^*$ is the optimal value of the problem, i.e., $f^*=\min_{z\in \R} f(z)$, and
$\bar x(0) = \frac{1}{n} \sum_{i=1}^n x_i^0.$
The scalars $\delta$ and $\lambda$ are functions of the graph sequence $G_{A^0}, G_{A^1}, G_{A^2}, \ldots$; in particular,
\begin{align*} \delta \geq  \frac{1}{n^{nB}} \quad \text{ and } \quad
\lambda \leq \left( 1 - \frac{1}{n^{nB}} \right)^{1/(nB)}.
\end{align*}
\end{enumerate} 
\end{theorem} 
We note that the lower bounds on $\delta$ and $\lambda$ can be refined when some additional structure is imposed on the underlying time-varying graphs.  The results of Theorem~\ref{mainthm} and their proofs can be found in~\cite{AAdirected} for a more general case when $f_i$ are defined over $\R^d$ with $d\ge 1$.
A better convergence rate can be obtained under the additional assumption that the objective functions $f_i$ are strongly convex; 
see~\cite{AAstrongly}.  

The first work to have employed Push-Sum decentralized averaging within a decentralized optimization methods is
\cite{tsianos2012consensus}, and it was further investigated in~\cite{rabbat_cdc2012,Tsianos2011,Tsianos2013}.
This work focused on static graphs, and it has been proposed as an alternative to the algorithm based on synchronous decentralized averaging over undirected graphs in order to avoid deadlocks and synchronization issues, among others. This work also described a decentralized method based on Push-Sum for multi-agent optimization problems with constraints by using Nesterov's dual-averaging approach. This Push-Sum consensus-based algorithm has been extended to the {\it subgradient-push} algorithm 
in~\cite{AAdirected,AAstrongly} that can deal with convex optimization problems over time-varying directed graphs. 
More recently, the paper~\cite{Sun2016} has extended the Push-Sum algorithm to a larger class of decentralized algorithms that are applicable to nonconvex objectives, convex constraint sets, and time-varying graphs.

References \cite{Xi2015,Zeng2015} combine EXTRA with the Push-Sum approach to 
produce the DEXTRA (Directed Extra-Push) algorithm for optimization over a directed graph.
It has been shown that DEXTRA converges at a geometric (R-linear) rate for a strongly convex objective function, 
but it requires a careful stepsize selection. It has been noted in~\cite{Xi2015} that 
the feasible region of stepsizes which guarantees this convergence rate can be empty in some cases.

\section{Extensions and Other Work on Decentralized Optimization}
\label{sec:extensions}
We discuss here some extension as well as other algorithms for minimizing the average sum  
$f(\cdot)=\frac{1}{n}\sum_{i=1}^{n} f_i(\cdot)$ in a decentralized manner.

\subsection{Extensions}
\subsubsection{Per-agent constraints}
When the nodes have a common convex and closed constraint set (known to each node) $X\subset\R^d$, the algorithms discussed in Section~\ref{sec:undirected} easily extend to handle the simple set constraints by performing projections on the set $X$.
For example, the decentralized subgradient method in Eq.~\eqref{distsub} can be modified to the following
update rule:
\begin{equation} \label{distsub-proj} 
x_i^{k+1} = \Pi_X\left[\sum_{j \in N_i^k} a_{ij}^k x_j^k - \alpha^k g_i^k\right], \end{equation} 
where $\Pi_X[\cdot]$ is the Euclidean projection on the set $X$. The subgradient $g_i^k$ of the function $f_i$ can be evaluated at the past iterate $x_i^k$ or at the point $\sum_{j \in N_i^k} a_{ij}^k x_j^k$.
Since the projection mapping $\Pi_X[\cdot]$
is non-expansive (i.e., $\|\Pi_X[x]- \Pi_X[y]\|_2\le \|x-y\|_2$ for all $x,y$), the convergence properties of the algorithm with projections remain the same as that of the algorithm without projections.

A more complicated case arises when the constraint set is given as the intersection of per-node constraint sets,
i.e., $X=\cap_{i=1}^n X_i$, where each $X_i$ is a convex closed set and only known to node $i$. In this case, the node $i$ update in Eq.~\eqref{distsub-proj} is modified by replacing $\Pi_X[\cdot]$ with $\Pi_{X_i}[\cdot]$, \an{thus resulting in the following updates
\begin{equation*} 
x_i^{k+1} = \Pi_{X_i}\left[\sum_{j \in N_i^k} a_{ij}^k x_j^k - \alpha^k g_i^k\right]. \end{equation*} 
The projections on the individual agents' constraint sets $X_i$ (instead of the true constraint set $X=\cap_{i=1}^n X_i$) introduce additional ``perturbations'', which can be controlled with the step-size $\alpha^k$, provided that the sets $X_i$ exhibit some form of regularity. Regularity is a condition requiring that the sum of the distances of a point from the individual sets $X_i$ is lower bounded by the distance of the point to the intersection of the sets,
i.e., $\sum_{i=1}^n \|x-\Pi_{X_i}\|^2_2\ge c\|x-\Pi_X\|^2_2$ for all $x$.
As a result of these additional perturbations coming from the sets $X_i$, the convergence analysis of the method is much more involved.} This algorithm, including random set-selections, has been studied 
in~\cite{NOP2010,SoominThesis,LN12b,LN2016} for synchronous updates over time-varying graphs 
and for (random) asynchronous updates over a static graph.
A variant of this algorithm (using the Laplacian formulation of the consensus problem) 
for decentralized optimization 
with decentralized constraints in noisy networks has been studied in~\cite{SN2010cdc,Kunal2011}.

\subsubsection{Effect of noise}
\an{We will discuss two possibilities, the case of (stochastic) noisy (sub)gradients and the case of noisy links. In the former case, the decentralized subgradient method proceeds by using stochastic (sub)gradients instead 
of subgradients. In particular, it assumes
the following form:
\begin{equation*} 
x_i^{k+1} = \sum_{j \in N_i^k} a_{ij}^k x_j^k - \alpha^k \tilde g_i^k(\omega), \end{equation*} 
where $\tilde g_i^k(\omega)$ is a stochastic vector (depending on a random variable $\omega$). As long as the stochastic subgradients have zero mean and bounded variance (the conditions typically needed for convergence of the centralized stochastic subgradient method) the decentralized method above can converge with probability 1 to a minimizers of $f(\cdot)$. In particular, if the conditions of Theorem~\ref{mainoptthm} are satisfied, and the stochastic subgradient errors are such that $\mathbb{E}_\omega\left[\tilde g_i^k(\omega)\mid x_{i}^k\right]=g_i^k$ for some subgradient $g_i^k$ and $\mathbb{E}_\omega\left[\|\tilde g_i^k(\omega)-g_i^k\|^2_2\mid x_{i}^{k}\right]\le \sigma^2$ for all $k$ and $i$,
then it can be seen that $\lim_{k\to\infty} x_i^k=x^*$
for all $i$ and for some $x^*\in X^*$. Such a result can be shown by incorporating the analysis of stochastic approximation methods with that of the decentralized subgradient method; the work addressing the decentralized stochastic methods can be found in~\cite{Ram2010,Kunal2011,LNR2017} for undirected time-varying graphs. Tighter bounds on the rate of convergence are obtained for a stochastic version of the distributed dual averaging algorithm in~\cite{tsianos2016efficient}. For the case of directed time-varying graphs, the push-sum based method is studied for the case of stochastic subgradients in~\cite{AAdirected}.}

\an{When the links are noisy, the agent $i$ may receive $x_j^k+\xi^k_{ij}$ instead of the actual quantity $x_j^k$ that was sent by its neighbor $j$, where $\xi^k_{ij}$ is a random link noise. The decentralized algorithm has the following form in this case:
\begin{equation*} 
x_i^{k+1} = \sum_{j \in N_i^k} a_{ij}^k \left(x_j^k +\xi_j^k\right)- \alpha^k \tilde g_i^k(\omega). \end{equation*}
Assuming that the noise process $\{\xi_{ij}^k\}$ has zero mean and bounded variance, one can show that all the iterate sequences $\{x_i^k\}$ converge to the same minimizer of $f(\cdot)$ almost surely (see for example~\cite{kunal-cdc10,Kunal2011}). 
Decentralized inference algorithms for general estimation problems (including nonlinear least squares) in stochastically time-varying networks under noisy gradient computation have been considered in~\cite{kar2011convergence,kar2012distributed,kar2014asymptotically}.}

\subsubsection{Random graphs}
\an{In the literature of the decentralized methods for multi-agent optimization, the graph sequence $G_1,G_2,\ldots$ is typically assumed to be externally given. The objective is to develop decentralized methods, given a graph sequence that constrains the agent communications. Under such a point of view, the algorithmic design does not address the question of designing the graph sequence, hence, does not optimize the network connectivity structure.}

\an{Another common assumption encountered in the literature is that the graph sequence $G_1,G_2,\ldots$ is deterministic.
The only work known to us that departures from such an assumption is the case when the graph sequence $\{G_\ell\}$ is independent and identically distributed (iid) random process. Specifically, each $G_\ell$ is a random realization drawn from a given distribution on the set of all possible graphs on $n$ nodes. In such a case, the connectivity assumption is imposed on the expected graph $\bar G =\mathbb{E}[G_\ell]$.
In this case, the decentralized subgradient becomes stochastic
\begin{equation*} 
x_i^{k+1} = \sum_{j \in N_i^k} a_{ij}^k x_j^k - \alpha^k \tilde g_i^k(\omega), \end{equation*}
where the neighbor sets $N_i^k$ are random. To ease the representation, the method is re-written as
\begin{equation} \label{rand-dec}
x_i^{k+1} = \sum_{j=1}^n a_{ij}^k x_j^k - \alpha^k \tilde g_i^k(\omega), \end{equation}
where $A^k$ is a stochastic (or double stochastic), and random with $\bar A =\mathbb{E}[A^\ell]$. Assuming that the graph $G_{\bar A}$ is undirected and connected, the above decentralized subgradient method has been studied in~\cite{Lobel2011}. A related version of the distributed dual averaging algorithm is studied in the setting of iid undirected graphs in \cite{Duchi2012}.}

\an{A special case of such a random iid graph sequence corresponds to the case when the agents use a random gossip or a random broadcast to communicate over a network. These random protocols have traditionally been used in network communication literature as protocols designed for asynchronous information exchange. They have also been used in design of decentralized multi-agent optimization methods, as discussed in the next subsection.}

\subsubsection{Asynchronous vs synchronous computations}
All the algorithms we discussed so far have been {\it synchronous} in the sense that all nodes update at the same time and also use the same stepsize $\alpha^k$ at iteration $k$. To accommodate the asynchronous updates 
and, also, allow that agents use different stepsizes, one may resort to a random gossip or broadcast communications, where
a random link is activated for communication (gossip) or a random node is activated to broadcast its information to the neighbors. 
\an{In this case, the decentralized method assumes the form as given in Eq.~\eqref{rand-dec}
where the matrix $A^k$ takes a particular form. Specifically, for the random gossip scheme, the underlying undirected graph $G$ is static and, at any time $k$, only one edge is activated at random, say the edge connecting agents $i_k$ and $j_k$. In this case, the matrix $A^k$ has the following form
\[A^k=I-\gamma (e_{i_k}-e_{j_k}(e_{i_k}-e_{j_k})^T,\]
where $\gamma\in(0,1)$, and $e_i$ denotes the unit-norm vector with $i$ entry equal to 1 and all other entries equal to 0. Each matrix $A^k$ is doubly stochastic, implying that the expected matrix $\bar A$ is also doubly stochastic.}

\an{In the case of a random broadcast, at every iteration $k$,
each node can be activated with probability $1/n$,
and the activated node broadcasts its value $x_i^k$ to all of the neighbors 
$j\in N_i$. Given that a node $i_k$ was activated at time $k$, 
the updates follow the rule given in Eq.~\eqref{rand-dec}, where 
\[A^k_{j,i_k}=\gamma,\quad A^k_{jj}=1-\gamma\qquad \hbox{for all\ } j\in N_{i_k},\]
\[A^k_{ii}=1\quad\hbox{for all $i\ne i_k$},\qquad A^k_{ij}=0\quad\hbox{otherwise},\]
where $\gamma\in(0,1)$. Here, the node $i_k$ and the nodes that are not its neighbors, $\ell\not\in N_{i_k}$, do not update. Only the neighbors of the node $i_k$ update.
The matrices $A^k$ are stochastic but not doubly stochastic. However, they have a special property (in expectation) that pushes the iterates toward a consensus (see~\cite{Nedic2011}), while the subgradients drive the iterates toward a minimizer of $f(\cdot)$, resulting in an asynchonous algorithm converging with probability 1.}

Consensus algorithms implemented in a network using a gossip-based or a broadcast-based communications have been 
studied in~\cite{Boyd-gossip,Aysal08,Aysal09,pieee11}, while 
a different consensus algorithm (the push-sum method) has been considered in~\cite{kempe03,benezit}.
A nonlinear gossip method is investigated in~\cite{borkar2014b}, while the survey paper~\cite{gossipsurvey} provides a detailed account of gossip algorithms for decentralized averaging and their applications to signal processing in sensor networks.

\subsection{Additional work on decentralized optimization}

A decentralized algorithm preserving an optimality condition at every iteration
has been proposed in~\cite{Lu}. Decentralized convex
optimization algorithms for weight-balanced directed graphs have been investigated in continuous-time~\cite{Gh-Cortes}. 

A different type of a decentralized algorithm for convex optimization has been proposed in~\cite{Li_Marden}, where each agent
keeps an estimate for all agents' decisions. This algorithm solves a problem where the agents have to minimize a global 
cost function $f(x_1,\ldots,x_m)$ while each agent $i$ can control only its variable $x_i$. 
The algorithm of~\cite{Li_Marden} has been recently extended to the online optimization setting 
in~\cite{NLR2016acc, LNR2016, LNR2017}.
Decentralized algorithms based on the augmented Lagrangian approach with gossip-type communications have been studied 
in~\cite{Jakovetic2011a}, and accelerated versions of decentralized gradient methods have been proposed and studied 
in~\cite{Jakovetic2011b}.
A consensus-based algorithm for solving problems with a separable
constraint structure and the use of primal-dual decentralized methods have been studied
in~\cite{ZhuM2012,ZhuMartinez2013}, 
\cite{Kunal2013}, while a decentralized primal-dual approach with perturbations have been explored in~\cite{HTC2014}. 
Work in~\cite{Wang} provides algorithms for centralized and decentralized convex optimization
from the control perspective, while~\cite{WLemon} considers an event-triggered decentralized optimization for sensor networks.
In~\cite{Burger}, a decentralized simplex algorithm has been developed for linear programming problems,
while a Newton-Raphson consensus-based method has been proposed in~\cite{Zanella} for decentralized convex problems.

Although our discussion has mainly focused on studying asymptotic rates of convergence of iterative decentralized optimization methods, in \cite{notarstefano2011} it was shown that a related approach based on consensus can solve general constrained abstract optimization problems in a finite number of iterations.

All of the work mentioned above relies on the use of state-independent weights, i.e., the weights that do not depend on the agents' iterates. 
A consensus-based algorithm employing state-dependent weights has been proposed and analyzed in~\cite{LobelOF2011}. 

Another popular decentralized approach for consensus optimization over a static network 
is the \emph{alternating direction method of multipliers} (ADMM). 
This method is based on an equivalent formulation of the consensus constraints. 
Unlike consensus-based (sub)-gradient method, which operates in the space of the primal-variables, 
the ADMM solves a corresponding Lagrangian dual problem 
(obtained by relaxing the equality constraints that are associated with consensus requirement).
Just as any dual method,
the ADMM is applicable to problems where the structure of the objective functions $f_i$ is simple enough so that the ADMM updates can be executed efficiently. The algorithm has the potential solve the problem with a geometric convergence rate, which requires global knowledge of some parameters including eigenvalues of a weight matrix associated with the graph.
A recent survey on the ADMM and its various applications is given in~\cite{admm}. 
The first work to address the development of decentralized ADMM over a network is~\cite{Schizas2008consensus1,Wei2012,Wei2013},
and it has been investigated in~\cite{Ling2014}, while its linear rate has been shown in~\cite{Shi2014admm}. \ao{For an explicit analysis of the relationship to network topology, see \cite{bento1, bento2}}.
In~\cite{SerhatAybat2015} the ADMM with linearization has been proposed for special 
composite optimization problems over graphs.

The work in~\cite{Xu2015,XuThesis} utilizes an adapt-then-combine (ATC) strategy~\cite{Sayed2013,Sayed2014} 
of dynamic weighted-average consensus approach~\cite{Zhu2010} to develop a distribute algorithm, termed Aug-DGM algorithm. This algorithm can be used over static directed or undirected graphs (but requires doubly stochastic matrix). 
The most interesting aspect of the Aug-DGM algorithm is that it can produce convergent iterates even when different agents use different (constant) stepsizes. 

Simultaneously and independently,
the idea of tracking the gradient averages through the use of consensus 
has been proposed in~\cite{Xu2015} for convex unconstrained problems 
and in~\cite{Lorenzo2015} for non-convex problems with convex constraints.
The work in~\cite{Lorenzo2015,Lorenzo2016icassp,Lorenzo2016} develops 
a large class of decentralized algorithms, referred to as NEXT, which
utilizes various ``function-surrogate modules" thus providing a great flexibility in its use and rendering 
a new class of algorithms that subsumes many of the existing decentralized algorithms. 
The work in~\cite{Lorenzo2016icassp,Lorenzo2016} and 
in~\cite{XuThesis} have also been proposed independently, with the former preceding the latter.
The algorithm framework of~\cite{Lorenzo2015,Lorenzo2016icassp,Lorenzo2016} is applicable to nonconvex problems with convex constraint sets over time-varying graphs, but requires the use of doubly stochastic matrices. This assumption was recently removed
in~\cite{Sun2016} by using column-stochastic matrices, which are more general than the degree-based column-stochastic matrices of the push-sum method. 
Simultaneously and independently,
the papers~\cite{Lorenzo2016} and~\cite{Tatarenko2015} have appeared to treat nonconvex problems over graphs. 
The work in~\cite{Tatarenko2015} proposes and analyzes a decentralized gradient method based on the push-sum consensus
in deterministic and stochastic setting for unconstrained problems.

\section{Conclusion and Open Problems} \label{sec:conclusion}
We have discussed decentralized optimization methods for minimizing the average of the nodes' objectives over graphs. We have considered undirected and directed time varying graphs, and computational models for solving consensus problem in such graphs. Then, we have discussed decentralized optimization algorithms that combine optimization techniques with decentralized averaging algorithms. We have also discussed extensions of the consensus-based approaches and other decentralized optimization algorithms.

In terms of algorithm scalability with the number $n$ of nodes, at present, it is an open question whether any improvement on the quadratic convergence time of 
Proposition~\ref{diff-graphs} is possible without an additional assumption about the knowledge of $n$ \ao{(recall that the assumption of knowing a reasonable upper bound on $n$ was made in Theorem~\ref{thm:accelerate} to show a linear scaling with $n$ in the convergence time)}. Also, it is not known whether a linear convergence-time scaling can be obtained for time-varying graphs.

Another question for future research is the implementation of decentralized algorithms with lower communication requirements. In particular, even broadcast based communications can be expensive, in terms of 
the power needed to broadcast in some sensor networks. 
A question is how to implement decentralized algorithms with fewer communications,
and what trade-offs are involved in such implementations. Some initial investigations along these lines were presented in~\cite{Tsianos2012communication} in the context of stochastic optimization, where progressively more time is spent calculating gradients between each round of communication as the number of iterations progresses. There remains much further work to be done along these lines. 

Finally, we remark that although there are well-understood lower bounds on the number of iterations required to achieve an $\epsilon$-optimal solution in the context of centralized convex optimization~\cite{nemyud,NesterovIntro}, much less is understood about the fundamental limits of decentralized optimization. Although bounds on the number of iterations for centralized algorithms carry over directly to synchronous decentralized algorithms, since any decentralized algorithm can always be emulated on a centralized processor, these results do not provide insight into how much communication is fundamentally required to reach consensus on an $\epsilon$-optimal solution. In communication-constrained settings (e.g., where network links have very low bandwidth), it remains an open question as to how many iterations may be required, and a related line of questioning would be to understand when there may be tradeoffs between communication and computation (e.g., to reach an $\epsilon$-optimal solution there may be algorithms which require significant computation and lower communication, or vice versa).

\section*{Acknowledgements}
M.R.~thanks Mido Assran for a careful reading and suggestions that improved this paper.

\bibliographystyle{IEEEtran}
\bibliography{alex,mike,directed-opt,distributed}

\end{document}

%% file: master-worker.tex
\begin{tikzpicture}[every node/.style={draw},thick]
	\def \r {2cm}
	\node[rectangle] (0) at (0,0) {Master};
	\foreach \i in {1,2,3}{
		\node[circle] (\i) at ({90 - (\i-1)/5*360}:\r) {$f_\i$};
	}
    \node[draw=none,fill=none] (4) at ({90 - (3/5*360)}:\r) {$\ddots$};
    \node[circle] (5) at ({90 - (4/5*360)}:\r) {$f_n$};
	\path[]
	\foreach \i in {1,2,3,5}{
		(0) edge[<->] (\i)
	};
    \draw[dashed,<->] (0) -- (4);
\end{tikzpicture}

%% file: fully-connected.tex
\begin{tikzpicture}[every node/.style={draw},thick]
	\def \r {2cm}
	\foreach \i in {1,2,3}{
		\node[circle] (\i) at ({90 - (\i-1)/5*360}:\r) {$f_\i$};
	}
    \node[draw=none,fill=none] (4) at ({90 - (3/5*360)}:\r) {$\ddots$};
    \node[circle] (5) at ({90 - (4/5*360)}:\r) {$f_n$};
	\path[]
    	(1) edge[<->] (2)
        (1) edge[<->] (3)
        (1) edge[<->] (5)
        (2) edge[<->] (3)
        (2) edge[<->] (5)
        (3) edge[<->] (5);
    \path[dashed] 
    	(1) edge[<->] (4)
        (2) edge[<->] (4)
        (3) edge[<->] (4)
        (5) edge[<->] (4);
\end{tikzpicture}

%% file: general-connected.tex
\begin{tikzpicture}[every node/.style={draw},thick]
	\def \r {2cm}
	\foreach \i in {1,2,3}{
		\node[circle] (\i) at ({90 - (\i-1)/5*360}:\r) {$f_\i$};
	}
    \node[draw=none,fill=none] (4) at ({90 - (3/5*360)}:\r) {$\ddots$};
    \node[circle] (5) at ({90 - (4/5*360)}:\r) {$f_n$};
	\path[]
    	(1) edge[<->] (2)
        (1) edge[<->] (3)
        (2) edge[<->] (5)
        (3) edge[<->] (5);
    \draw[dashed,<->] (2) -- (4) -- (3);
\end{tikzpicture}

%% file: path-graph.tex
\begin{tikzpicture}[every node/.style={circle,draw,fill=black!50},thick]
	\foreach \i in {1,...,5}{
		\node (\i) at (\i, 0) {};
	}
	\path[] (1) edge (2)
			(2) edge (3)
			(3) edge (4)
			(4) edge (5);
\end{tikzpicture}

%% file: grid-graph.tex
\begin{tikzpicture}[every node/.style={circle,draw,fill=black!50},thick]
	\foreach \i in {1,...,4}{
		\foreach \j in {1,...,4}{
			\node (\i-\j) at (\i, \j) {};
		}
	}
	\foreach \i in {1,...,4}{
		\path[]
		\foreach \j/\k in {1/2,2/3,3/4}{
			(\i-\j) edge (\i-\k)
		};
	}
	\foreach \i in {1,...,4}{
		\path[]
		\foreach \j/\k in {1/2,2/3,3/4}{
			(\j-\i) edge (\k-\i)
		};
	}
\end{tikzpicture}

%% file: star-graph.tex
\begin{tikzpicture}[every node/.style={circle,draw,fill=black!50},thick]
	\def \r {1cm}
	\node (1) at (0,0) {};
	\foreach \i in {2,...,6}{
		\node (\i) at ({(\i-1)/5*360 + 90}:\r) {};
	}	
	\path[]
	\foreach \i in {2,...,6}{
		(1) edge (\i)
	};
\end{tikzpicture}

%% file: two-star-graph.tex
\begin{tikzpicture}[every node/.style={circle,draw,fill=black!50},thick]
	\def \r {1cm}
	\node (1) at (0,0) {};
	\foreach \i in {2,...,6}{
		\node (\i) at ({(\i-1)/5*360 + 90}:\r) {};
	}	
	\path[]
	\foreach \i in {2,...,6}{
		(1) edge (\i)
	};
	\node (7) at (3,0) {};
	\foreach \i in {8,...,12}{
		\path (7) ++ ({(\i-1)/5*360 + 90}:\r) node (\i) {};
	}	
	\path[]
	\foreach \i in {8,...,12}{
		(7) edge (\i)
	};
	\path (1) edge (7);
\end{tikzpicture}

%% file: multiagentopt.bbl
\begin{thebibliography}{100}
\providecommand{\url}[1]{#1}
\csname url@samestyle\endcsname
\providecommand{\newblock}{\relax}
\providecommand{\bibinfo}[2]{#2}
\providecommand{\BIBentrySTDinterwordspacing}{\spaceskip=0pt\relax}
\providecommand{\BIBentryALTinterwordstretchfactor}{4}
\providecommand{\BIBentryALTinterwordspacing}{\spaceskip=\fontdimen2\font plus
\BIBentryALTinterwordstretchfactor\fontdimen3\font minus
  \fontdimen4\font\relax}
\providecommand{\BIBforeignlanguage}[2]{{%
\expandafter\ifx\csname l@#1\endcsname\relax
\typeout{** WARNING: IEEEtran.bst: No hyphenation pattern has been}%
\typeout{** loaded for the language `#1'. Using the pattern for}%
\typeout{** the default language instead.}%
\else
\language=\csname l@#1\endcsname
\fi
#2}}
\providecommand{\BIBdecl}{\relax}
\BIBdecl

\bibitem{Bertsekas2003}
D.~Bertsekas, A.~Nedi\'{c}, and A.~Ozdaglar, \emph{Convex Analysis and
  Optimization}.\hskip 1em plus 0.5em minus 0.4em\relax Athena Scientific,
  2003.

\bibitem{BoydVandenberghe}
S.~Boyd and L.~Vandenberghe, \emph{Convex Optimization}.\hskip 1em plus 0.5em
  minus 0.4em\relax Cambridge Univeristy Press, 2004.

\bibitem{NocedalWright}
J.~Nocedal and S.~Wright, \emph{Numerical Optimization}, 2nd~ed.\hskip 1em plus
  0.5em minus 0.4em\relax Springer, 2006.

\bibitem{Tsi1986}
J.~Tsitsiklis, D.~Bertsekas, and M.~Athans, ``Distributed asynchronous
  deterministic and stochastic gradient optimization algorithms,'' \emph{IEEE
  Transactions on Automatic Control}, vol.~31, no.~9, pp. 803 -- 812, Sep.
  1986.

\bibitem{tsitsiklisPhDThesis}
J.~Tsitsiklis, ``Problems in decentralized decision making and computation,''
  Ph.D. dissertation, M.I.T., Nov. 1984.

\bibitem{bertsekasParallel}
D.~P. Bertsekas and J.~N. Tsitsiklis, \emph{Parallel and distributed
  computation: numerical methods}, 1st~ed.\hskip 1em plus 0.5em minus
  0.4em\relax Upper Saddle River, NJ, USA: Prentice-Hall, Inc., 1989.

\bibitem{cao2013overview}
Y.~Cao, W.~Yu, W.~Ren, and G.~Chen, ``An overview of recent progress in the
  study of distributed multi-agent coordination,'' \emph{IEEE Transactions on
  Industrial Informatics}, vol.~9, no.~1, pp. 427--438, Feb. 2013.

\bibitem{barrenetxea2008sensorscope}
G.~Barrenetxea, F.~Inglerest, G.~Schaefer, and M.~Vetterli, ``Wireless sensor
  networks for environmental monitoring: The {SensorScope} experience,'' in
  \emph{Proc. IEEE Intl. Zurich Seminar on Communications}, Zurich,
  Switzerland, Mar. 2008.

\bibitem{rabbat2004distributed}
M.~Rabbat and R.~Nowak, ``Distributed optimization in sensor networks,'' in
  \emph{Proc. ACM/IEEE Intl. Conf. on Information Processing in Sensor Networks
  (IPSN)}, Berkeley, CA, USA, Apr. 2004, pp. 20--27.

\bibitem{schizas2008consensus2}
I.~Schizas, G.~Giannakis, S.~Roumeliotis, and A.~Ribeiro, ``Consensus in ad hoc
  {WSNs} with noisy links - part ii: Distributed estimation and smoothing of
  random signals,'' \emph{IEEE Transactions on Signal Processing}, vol.~56,
  no.~4, pp. 1650--1666, Apr. 2008.

\bibitem{cattivelli2010diffusion}
F.~Cattivelli and A.~Sayed, ``Diffusion {LMS} strategies for distributed
  estimation,'' \emph{IEEE Transactions on Signal Processing}, vol.~58, no.~3,
  pp. 1035--1048, Mar. 2010.

\bibitem{kar2011convergence}
S.~Kar and J.~Moura, ``Convergence rate analysis of distributed gossip (linear
  parameter) estimation: Fundamental limits and tradeoffs,'' \emph{IEEE Journal
  of Selected Topics in Signal Processing}, vol.~5, no.~4, pp. 674--690, Aug.
  2011.

\bibitem{kar2012distributed}
S.~Kar, J.~Moura, and K.~Ramanan, ``Distributed parameter estimation in sensor
  networks: Nonlinear observation models and imperfect communication,''
  \emph{IEEE Transactions on Information Theory}, vol.~58, no.~6, pp.
  3575--3605, Jun. 2012.

\bibitem{kar2014asymptotically}
S.~Kar and J.~Moura, ``Asymptotically efficient distributed estimation with
  exponential family statistics,'' \emph{IEEE Transactions on Information
  Theory}, vol.~60, no.~8, pp. 4811--4831, Aug. 2014.

\bibitem{understandingML}
S.~{Shalev-Shwartz} and S.~{Ben-David}, \emph{Understanding Machine Learning:
  From Theory to Algorithms}.\hskip 1em plus 0.5em minus 0.4em\relax Cambridge
  University Press, 2014.

\bibitem{ram2009new}
S.~Ram, A.~Nedic, and V.~Veeravalli, ``A new class of distributed optimization
  algorithms: Application to regression of distributed data,''
  \emph{Optimization Methods and Software}, vol.~27, no.~1, pp. 71--88, 2009.

\bibitem{tsianos2012consensus}
K.~Tsianos, S.~Lawlor, and M.~Rabbat, ``Consensus-based distributed
  optimization: Practical issues and applications in large-scale machine
  learning,'' in \emph{50th Allerton Conference on Communication, Control, and
  Computing}, 2012.

\bibitem{chen2012diffusion}
J.~Chen and A.~H. Sayed, ``Diffusion adaptation strategies for distributed
  optimization and learning over networks,'' \emph{IEEE Transactions on Signal
  Processing}, vol.~60, no.~8, pp. 4289--4305, Aug. 2012.

\bibitem{Duchi2012}
J.~Duchi, A.~Agarwal, and M.~Wainwright, ``Dual averaging for distributed
  optimization: Convergence analysis and network scaling,'' \emph{IEEE
  Transactions on Automatic Control}, vol.~57, no.~3, pp. 592--606, 2012.

\bibitem{Jakovetic2011a}
D.~Jakoveti\'c, J.~Xavier, and J.~Moura, ``Cooperative convex optimization in
  networked systems: Augmented lagrangian algorithms with directed gossip
  communication,'' \emph{IEEE Transactions on Signal Processing}, vol.~59,
  no.~8, pp. 3889--3902, 2011.

\bibitem{cortes2006robust}
J.~Cort\'{e}s, S.~Mart\'{i}nez, and F.~Bullo, ``Robust rendezvous for mobile
  autonomous agents via proximity graphs in arbitrary dimensions,'' \emph{IEEE
  Transactions on Automatic Control}, vol.~51, no.~8, pp. 1289--1298, Aug.
  2006.

\bibitem{zavlanos2013network}
M.~Zavlanos, A.~Ribeiro, and G.~Pappas, ``Network integrity in mobile robotic
  networks,'' \emph{IEEE Transactions on Automatic Control}, vol.~58, no.~1,
  pp. 3--18, Jan. 2013.

\bibitem{jadbabaie03}
A.~Jadbabaie, J.~Lin, and A.~Morse, ``Coordination of groups of mobile
  autonomous agents using nearest neighbor rules,'' \emph{IEEE Transactions on
  Automatic Control}, vol.~48, no.~6, pp. 988--1001, 2003.

\bibitem{Nedic09a}
A.~Nedi\'c and A.~Ozdaglar, ``Distributed subgradient methods for multi-agent
  optimization,'' \emph{IEEE Transactions on Automatic Control}, vol.~54,
  no.~1, pp. 48--61, 2009.

\bibitem{durett}
R.~Durrett, \emph{Random Graph Dynamics}.\hskip 1em plus 0.5em minus
  0.4em\relax Cambridge University Press, 2007.

\bibitem{PenroseRGGbook}
M.~Penrose, \emph{Random Geometric Graphs}.\hskip 1em plus 0.5em minus
  0.4em\relax Oxford University Press, 2003.

\bibitem{lincons2}
A.~Olshevsky, ``Linear time average consensus on fixed graphs,'' in
  \emph{Proceedings of NecSys, the 3rd IFAC Workshop on Distributed Estimation
  and Control in Networked Systems}, 2015.

\bibitem{lpw}
E.~W. D.~Levin, Y.~Peres, \emph{Markov Chains and Mixing Times}.\hskip 1em plus
  0.5em minus 0.4em\relax American Mathematical Society, 2009.

\bibitem{commute}
A.~K. Chandra, P.~Raghavan, W.~Ruzzo, R.~Smolensky, and P.~Tiwari, ``The
  electrical resistance of a graph captures its commute and cover times,''
  \emph{Computational Complexity}, vol.~6, no.~4, pp. 312--340, 1996.

\bibitem{avin}
C.~Avin and G.~Ercal, ``On the cover time and mixing time of random geometric
  graphs,'' \emph{Theoretical Computer Science}, vol. 380, no.~4, pp. 2--22,
  2007.

\bibitem{lincons1}
A.~Olshevsky, ``Linear time average consensus and distributed optimization on
  fixed graphs,'' \emph{SIAM Journal on Control and Optimization}, vol.~55,
  no.~6, pp. 3990--4014, 2017.

\bibitem{Xiao2004}
L.~Xiao and S.~Boyd, ``Fast linear iterations for distributed averaging,''
  \emph{Systems and Control Letters}, vol.~53, pp. 65--78, 2004.

\bibitem{Boyd2006randomized}
S.~Boyd, A.~Ghosh, B.~Prabhakar, and D.~Shah, ``Randomized gossip algorithms,''
  \emph{IEEE Transactions on Information Theory}, vol.~52, no.~6, pp.
  2508--2530, Jun. 2006.

\bibitem{Spielman2017}
D.~Spielman, ``Graphs, vectors, and matrices,'' \emph{Bulletin of the American
  Mathematical Society}, vol.~54, no.~1, pp. 45--61, Jan. 2017.

\bibitem{olfati2007algebraic}
R.~{Olfati-Saber}, ``Algebraic connectivity ratio of {Ramanujan} graphs,'' in
  \emph{Proc. IEEE American Conf. on Control}, New York, USA, Jul. 2007.

\bibitem{kar2008topology}
S.~Kar, S.~Aldosari, and J.~Moura, ``Topology for distributed inference in
  graphs,'' \emph{IEEE Transactions on Signal Processing}, vol.~56, no.~6, pp.
  2609--2613, Jun. 2008.

\bibitem{tsitsiklis1986decentralized}
J.~Tsitsiklis, D.~Bertsekas, and M.~Athans, ``Distributed asynchronous
  deterministic and stochastic gradient optimization algorithms,'' \emph{IEEE
  Transactions on Automatic Control}, vol.~31, no.~9, pp. 803--812, 1986.

\bibitem{varga}
R.~Varga, \emph{Matrix iterative analysis}.\hskip 1em plus 0.5em minus
  0.4em\relax Springer, 2000.

\bibitem{nest}
Y.~Nesterov, ``A method of solving a convex optimization problem with
  convergence rate $o(1/k^2)$,'' \emph{Soviet Mathematics Doklady}, vol.~27,
  no.~2, pp. 372--376, 1983.

\bibitem{wikipedia}
\url{https://de.wikipedia.org/wiki/Subdifferential}, image released into the
  public domain by Felix Reidel.

\bibitem{polyak}
B.~Polyak, ``A general method for solving extremum problems,'' \emph{Soviet
  Mathematical Doklady}, vol.~8, no.~3, pp. 593--597, 1967.

\bibitem{shor}
N.~Shor, \emph{Minimization methods for Nondifferentiable Functions}.\hskip 1em
  plus 0.5em minus 0.4em\relax Berlin: Translated from Russian by K.C. Kiwiel
  and A. Ruszczynski, Springer, 1985.

\bibitem{polyakbook}
B.~Polyak, \emph{Introduction to Optimisation}.\hskip 1em plus 0.5em minus
  0.4em\relax New York: Optimization Software, Inc., 1987.

\bibitem{AAdirected}
A.~Nedi\'c and A.~Olshevsky, ``Distributed optimization over time-varying
  directed graphs,'' \emph{IEEE Transactions on Automatic Control}, vol.~60,
  no.~3, pp. 601--615, 2015.

\bibitem{Ram2012}
S.~S. Ram, A.~Nedi\'c, and V.~V. Veeravalli, ``A new class of distributed
  optimization algorithms: application to regression of distributed data,''
  \emph{Optimization Methods and Software}, vol.~27, no.~1, pp. 71--88, 2012.

\bibitem{Nedic2011}
A.~Nedi\'c, ``Random projection algorithms for convex minimization problems,''
  \emph{Mathematical Programming, Series {B}}, vol. 129, pp. 225--253, 2011.

\bibitem{Ram2010}
S.~Ram, A.~Nedi\'c, and V.~Veeravalli, ``{Distributed Stochastic Subgradient
  Projection Algorithms for Convex Optimization},'' \emph{Journal of
  Optimization Theory and Applications}, vol. 147, pp. 516--545, 2010.

\bibitem{nemyud}
A.~Nemirovski and D.~Yudin, \emph{Problem Complexity and Method Efficiency in
  Optimization}.\hskip 1em plus 0.5em minus 0.4em\relax John Wiley \& Sons,
  1983.

\bibitem{extra}
W.~Shi, Q.~Ling, G.~Wu, and W.~Yin, ``{EXTRA:} an exact first order algorithm
  for decentralized consensus optimization,'' \emph{SIAM Journal on
  Optimization}, vol.~25, no.~2, pp. 944--966, 2015.

\bibitem{diging}
A.~Nedi\'c, A.~Olshevsky, and W.~Shi, ``Achieving geometric convergence for
  distributed optimization over time-varying graphs,'' 2016, accepted at {\it
  SIAM Journal on Optimization} 2017, available on arXiv at
  https://arxiv.org/abs/1607.03218.

\bibitem{ZhuM2012}
M.~Zhu and S.~Mart\'{i}nez, ``On distributed convex optimization under
  inequality and equality constraints,'' \emph{IEEE Transactions on Automatic
  Control}, vol.~57, no.~1, pp. 151--164, 2012.

\bibitem{Xu2015}
J.~Xu, S.~Zhu, Y.~Soh, and L.~Xie, ``Augmented distributed gradient methods for
  multi-agent optimization under uncoordinated constant stepsizes,'' in
  \emph{Proceedings of the 54th IEEE Conference on Decision and Control (CDC)},
  2015, pp. 2055--2060.

\bibitem{XuThesis}
J.~Xu, ``Augmented distributed optimization for networked systems,'' Ph.D.
  dissertation, School of Electrical and Electronic Engineering, Nanyang
  Technological University, 2016.

\bibitem{Lorenzo2015}
P.~D. Lorenzo and G.~Scutari, ``Distributed nonconvex optimization over
  networks,'' in \emph{Proceedings of IEEE International Conference on
  Computational Advances in Multi-Sensor Adaptive Processing (CAMSAP 2015),
  Dec.\ 13--16, 2015, Cancun, Mexico}, 2015.

\bibitem{Lorenzo2016icassp}
------, ``Distributed nonconvex optimization over time-varying networks,'' in
  \emph{Proceedings of IEEE International Conference on Acoustics, Speech, and
  Signal Processing (ICASSP 16), March 20--25, 2016, Shanghai, China}, 2016.

\bibitem{Lorenzo2016}
------, ``{NEXT}: In-network nonconvex optimization,'' \emph{IEEE Transactions
  on Signal and Information Processing over Networks}, vol.~2, no.~2, pp.
  120--136, 2016.

\bibitem{lina1}
G.~Qu and N.~Li, ``Accelerated distributed nesterov gradient descent,''
  preprint. Available at \url{https://arxiv.org/abs/1705.07176}.

\bibitem{lina2}
------, ``Harnessing smoothness to accelerate distributed optimization,'' 2017,
  to appear in \textit{IEEE Transactions on Control of Network Systems}.

\bibitem{julien-john}
J.~M. Hendrickx and J.~N. Tsitsiklis, ``Fundamental limitations for anonymous
  distributed systems with broadcast communications,'' in \emph{Proceedings of
  the 53rd Annual Allerton Conference on Communication, Control, and
  Computing}, 2015.

\bibitem{Charalambous2016decentralized}
T.~Charalambous, M.~Rabbat, M.~Johansson, and C.~Hadjicostis, ``Distributed
  finite-time computation of digraph parameters: Left-eigenvector, out-degree
  and spectrum,'' \emph{IEEE Transactions on Control of Network Systems},
  vol.~3, no.~2, pp. 137--148, 2016.

\bibitem{kempe03}
D.~Kempe, A.~Dobra, and J.~Gehrke, ``Gossip-based computation of aggregate
  information,'' in \emph{Proceedings of the 44th Annual IEEE Symposium on
  Foundations of Computer Science}, 2003, pp. 482--491.

\bibitem{benezit}
F.~Benezit, V.~Blondel, P.~Thiran, J.~Tsitsiklis, and M.~Vetterli, ``Weighted
  gossip: distributed averaging using non-doubly stochastic matrices,'' in
  \emph{Proceedings of the 2010 IEEE International Symposium on Information
  Theory}, Jun. 2010.

\bibitem{dominguez-energy}
A.~Dominguez-Garcia and C.~Hadjicostis, ``Distributed algorithms for control of
  demand responses and distributed energy resources,'' in \emph{Proceedings of
  the 50th IEEE Conference on Decision and Control and European Control
  Conference}, Dec 2011, pp. 27--32.

\bibitem{hadjicostis2014average}
C.~N. Hadjicostis and T.~Charalambous, ``Average consensus in the presence of
  delays in directed graph topologies,'' \emph{IEEE Transactions on Automatic
  Control}, vol.~59, no.~3, pp. 763--768, 2014.

\bibitem{hadjicostis2016robust}
C.~N. Hadjicostis, N.~H. Vaidya, and A.~D. Dom{\'\i}nguez-Garc{\'\i}a, ``Robust
  distributed average consensus via exchange of running sums,'' \emph{IEEE
  Transactions on Automatic Control}, vol.~61, no.~6, pp. 1492--1507, 2016.

\bibitem{dominguez}
A.~Dominguez-Garcia and C.~Hadjicostis, ``Distributed strategies for average
  consensus in directed graphs,'' in \emph{Proceedings of the 50th IEEE
  Conference on Decision and Control and European Control Conference}, Dec
  2011.

\bibitem{AAstrongly}
A.~Nedi\'c and A.~Olshevsky, ``Stochastic gradient-push for strongly convex
  functions on time-varying directed graphs,'' \emph{IEEE Transactions on
  Automatic Control}, vol.~61, no.~12, pp. 3936 -- 3947, 2016.

\bibitem{rabbat_cdc2012}
K.~Tsianos, S.~Lawlor, and M.~Rabbat, ``Push-sum distributed dual averaging for
  convex optimization,'' in \emph{Proceedings of the IEEE Conference on
  Decision and Control}, 2012.

\bibitem{Tsianos2011}
K.~Tsianos and M.~Rabbat, ``Distributed consensus and optimization under
  communication delays,'' in \emph{Proc. of Allerton Conference on
  Communication, Control, and Computing}, 2011, pp. 974Ð--982.

\bibitem{Tsianos2013}
K.~Tsianos, ``The role of the network in distributed optimization algorithms:
  Convergence rates, scalability, communication / computation tradeoffs and
  communication delays,'' Ph.D. dissertation, McGill University, Dept. of
  Electrical and Computer Engineering, 2013.

\bibitem{Sun2016}
Y.~Sun, G.~Scutari, and D.~Palomar, ``Distributed nonconvex multiagent
  optimization over time-varying networks,'' in \emph{Proc. of the Asilomar
  Conference on Signals, Systems, and Computers}, Pacific Grove, CA, USA, Nov.
  2016.

\bibitem{Xi2015}
C.~Xi and U.~Khan, ``On the linear convergence of distributed optimization over
  directed graphs,'' 2015, available on arXiv at
  http://arxiv.org/abs/1510.02149.

\bibitem{Zeng2015}
J.~Zeng and W.~Yin, ``Extrapush for convex smooth decentralized optimization
  over directed networks,'' 2015, available on arXiv at
  http://arxiv.org/abs/1511.02942.

\bibitem{NOP2010}
A.~Nedi\'c, A.~Ozdaglar, and P.~Parrilo, ``Constrained consensus and
  optimization in multi-agent networks,'' \emph{IEEE Transactions on Automatic
  Control}, vol.~55, no.~4, pp. 922 --938, April 2010.

\bibitem{SoominThesis}
S.~Lee, ``Optimization over networks: {E}fficient algorithms and analysis,''
  Ph.D. dissertation, Department of Electrical and Computer Engineering,
  University of Illinois at Urbana-Champaign, 2013.

\bibitem{LN12b}
S.~Lee and A.~Nedi\'c, ``Distributed random projection algorithm for convex
  optimization,'' \emph{IEEE Journal on Selected Topics in Signal Processing},
  vol.~48, no.~6, pp. 988--1001, 2012.

\bibitem{LN2016}
------, ``Asynchronous gossip-based random projection algorithms over
  networks,'' \emph{IEEE Transactions on Automatic Control}, vol.~61, no.~4,
  pp. 953--968, 2016.

\bibitem{SN2010cdc}
K.~Srivastava, A.~Nedi\'c, and D.~Stipanovi\'c, ``Distributed constrained
  optimization over noisy networks,'' in \emph{Proceedings of the 49th IEEE
  Conference on Decision and Control (CDC)}, Dec. 2010, pp. 1945 --1950.

\bibitem{Kunal2011}
K.~Srivastava and A.~Nedi\'c, ``Distributed asynchronous constrained stochastic
  optimization,'' \emph{IEEE Journal of Selected Topics in Signal Processing},
  vol.~5, no.~4, pp. 772--790, 2011.

\bibitem{LNR2017}
S.~Lee, A.~Nedi\'c, and M.~Raginsky, ``Stochastic dual-averaging for
  decentralized online optimization on time-varying communication graphs,'' to
  appear, accepted for publication in {\it IEEE Transactions on Automatic
  Control}, January 2017.

\bibitem{tsianos2016efficient}
K.~Tsianos and M.~Rabbat, ``Efficient distributed online prediction and
  stochastic optimization with approximate distributed averaging,'' \emph{IEEE
  Transactions on Signal and Information Processing Over Networks}, vol.~2,
  no.~4, pp. 489--506, Dec. 2016.

\bibitem{kunal-cdc10}
K.~Srivastava, A.~Nedi\'c, and D.~Stipanovic, ``Distributed constrained
  optimization over noisy networks,'' in \emph{Proceedings of the 49th IEEE
  Conference on Decision and Control}, 2010, pp. 1945--1950.

\bibitem{Lobel2011}
I.~Lobel and A.~Ozdaglar, ``Distributed subgradient methods for convex
  optimization over random networks,'' \emph{IEEE Transactions on Automatic
  Control}, vol.~56, no.~6, pp. 1291 --1306, June 2011.

\bibitem{Boyd-gossip}
S.~Boyd, A.~Ghosh, B.~Prabhakar, and D.~Shah, ``Gossip algorithms: Design,
  analysis, and applications,'' in \emph{Proceedings of IEEE INFOCOM}, vol.~3,
  2005, pp. 1653--1664.

\bibitem{Aysal08}
T.~Aysal, M.~Yildiz, A.~Sarwate, and A.~Scaglione, ``Broadcast gossip
  algorithms: {D}esign and analysis for consensus,'' in \emph{Proceedings of
  the 47th IEEE Conference on Decision and Control}, 2008, pp. 4843--4848.

\bibitem{Aysal09}
T.~Aysal, M.~Yildriz, A.~Sarwate, and A.~Scaglione, ``Broadcast gossip
  algorithms for consensus,'' \emph{IEEE Transactions on Signal processing},
  vol.~57, pp. 2748--2761, 2009.

\bibitem{pieee11}
J.~Liu, S.~Mou, A.~Morse, B.~Anderson, and C.~Yu, ``Deterministic gossiping,''
  \emph{Proceedings of the IEEE}, vol.~99, no.~9, pp. 1505--1524, 2011.

\bibitem{borkar2014b}
A.~Mathkar and V.~Borkar, ``Nonlinear gossip,'' \emph{SIAM Journal on Control
  and Optimization}, vol.~54, no.~3, pp. 1535--1557, 2016.

\bibitem{gossipsurvey}
A.~Dimakis, S.~Kar, J.~Moura, M.~Rabbat, and A.~Scaglione, ``Gossip algorithms
  for distributed signal processing,'' \emph{Proceedings of the IEEE}, vol.~98,
  no.~11, pp. 1847--1864, 2010.

\bibitem{Lu}
J.~Lu and C.~Tang, ``Zero-gradient-sum algorithms for distributed convex
  optimization: the continuous-time case,'' \emph{IEEE Transactions on
  Automatic Control}, vol.~57, no.~9, pp. 2348--2354, 2012.

\bibitem{Gh-Cortes}
B.~Gharesifard and J.~Cortes, ``Distributed continuous-time convex optimization
  on weight-balanced digraphs,'' 2012, preprint, available at
  http://arxiv.org/abs/1204.0304.

\bibitem{Li_Marden}
N.~Li and J.~R. Marden, ``Designing games for distributed optimization,''
  \emph{IEEE Journal on Selected Topics in Signal Processing}, vol.~7, no.~2,
  pp. 230--242, 2013.

\bibitem{NLR2016acc}
A.~Nedi\'c, S.~Lee, and M.~Raginsky, ``Decentralized online optimization with
  global objectives and local communication,'' in \emph{Proceedings of the 2016
  American Control Conference (ACC), Boston, MA, July 6--8, 2016}, 2016, pp.
  4497--4503.

\bibitem{LNR2016}
S.~Lee, A.~Nedi\'c, and M.~Raginsky, ``Coordinate dual averaging for
  decentralized online optimization with nonseparable global objectives,'' to
  appear, accepted in {\it IEEE Transactions on Control of Network Systems},
  May 2016.

\bibitem{Jakovetic2011b}
D.~Jakoveti\'c, J.~Xavier, and J.~Moura, ``Fast distributed gradient methods,''
  \emph{IEEE Transactions on Automatic Control}, vol.~59, no.~5, pp.
  1131--1146, 2014.

\bibitem{ZhuMartinez2013}
M.~Zhu and S.~Mart\'{i}nez, ``An approximate dual subgradient algorithm for
  distributed non-convex constrained optimization,'' \emph{IEEE Transactions on
  Automatic Control}, vol.~58, no.~6, pp. 1534--1539, 2013.

\bibitem{Kunal2013}
K.~Srivastava, A.~Nedi\'c, and D.~Stipanovic, ``Distributed bregman-distance
  algorithms for min-max optimization,'' in \emph{Agent-Based
  Optimization}.\hskip 1em plus 0.5em minus 0.4em\relax Springer Studies in
  Computational Intelligence (SCI), 2013, pp. 143--174.

\bibitem{HTC2014}
T.-H. Chang, A.~Nedi\'c, and A.~Scaglione, ``Distributed constrained
  optimization by consensus-based primal-dual perturbation method,'' \emph{IEEE
  Transactions on Automatic Control}, vol.~59, no.~6, pp. 1524--1538, 2014.

\bibitem{Wang}
J.~Wang and N.~Elia, ``A control perspective for centralized and distributed
  convex optimization,'' in \emph{Proceedings of the IEEE Conference on
  Decision and Control, (Florida, USA)}, 2011, pp. 3800--3805.

\bibitem{WLemon}
P.~Wan and M.~Lemmon, ``Event-triggered distributed optimization in sensor
  networks,'' in \emph{Symposium on Information Processing of Sensor Networks,
  (San Francisco, CA)}, 2009, pp. 49--60.

\bibitem{Burger}
M.~B{\"u}rger, G.~Notarstefano, F.~Bullo, and F.~Allg{\"o}wer, ``A distributed
  simplex algorithm for degenerate linear programs and multi-agent
  assignments,'' \emph{Automatica}, vol.~48, no.~9, pp. 2298--2304, 2012.

\bibitem{Zanella}
F.~Zanella, D.~Varagnolo, A.~Cenedese, G.~Pillonetto, and L.~Schenato,
  ``Newton-raphson consensus for distributed convex optimization,'' in
  \emph{Proceedings of the IEEE Conference on Decision and Control, (Florida,
  USA)}, 2011, pp. 5917--5922.

\bibitem{notarstefano2011}
G.~Notarstefano and F.~Bullo, ``Distributed abstract optimization via
  constraints consensus: Theory and applications,'' \emph{IEEE Transactions on
  Automatic Control}, vol.~56, no.~10, pp. 2247--2261, 2011.

\bibitem{LobelOF2011}
I.~Lobel, A.~Ozdaglar, and D.~Feijer, ``Distributed multi-agent optimization
  with state-dependent communication,'' \emph{Mathematical Programming}, vol.
  129, no.~2, pp. 255--284, 2011.

\bibitem{admm}
S.~Boyd, N.~Parikh, E.~Chu, B.~Peleato, and J.~Eckstein, ``Distributed
  optimization and statistical learning via the alternating direction method of
  multipliers,'' \emph{Foundations and Trends in Machine Learning}, vol.~3,
  no.~1, pp. 1--122, 2010.

\bibitem{Schizas2008consensus1}
I.~Schizas, A.~Ribeiro, and G.~Giannakis, ``Consensus in ad hoc {WSNs} with
  noisy links---part {I}: Distributed estimation of deterministic signals,''
  \emph{IEEE Transactions on Signal Processing}, vol.~56, no.~1, pp. 350--364,
  Jan. 2008.

\bibitem{Wei2012}
E.~Wei and A.~Ozdaglar, ``Distributed alternating direction method of
  multipliers,'' in \emph{Proceedings of the 51st IEEE Conference on Decision
  and Control and European Control Conference}, 2012, pp. 5445--5450.

\bibitem{Wei2013}
------, ``On the {O}(1/k) convergence of asynchronous distributed alternating
  direction method of multipliers,'' in \emph{Proceedings of IEEE Global
  Conference on Signal and Information Processing}, 2013, pp. 551--554.

\bibitem{Ling2014}
Q.~Ling and A.~Ribeiro, ``Decentralized dynamic optimization through the
  alternating direction method of multiplier,'' \emph{IEEE Transactions on
  Signal Processing}, vol.~62, no.~5, pp. 1185--1197, 2014.

\bibitem{Shi2014admm}
W.~Shi, Q.~Ling, K.~Yuan, G.~Wu, and W.~Yin, ``On the linear convergence of the
  {ADMM} in decentralized consensus optimization,'' \emph{IEEE Transactions on
  Signal Processing}, vol.~62, no.~7, pp. 1750 -- 1761, 2014.

\bibitem{bento1}
G.~Franca and J.~Bento, ``Markov chain lifting and distributed admm,''
  \emph{IEEE Signal Processing Letters}, vol.~24, no.~3, pp. 294--298, 2017.

\bibitem{bento2}
------, ``How is distributed admm affected by network topology?''
  \url{https://arxiv.org/abs/1710.00889}.

\bibitem{SerhatAybat2015}
N.~Aybat, Z.~Wang, T.~Lin, and S.~Ma, ``Distributed linearized alternating
  direction method of multipliers for composite convex consensus
  optimization,'' 2015, preprint, available on arxiv at
  http://arxiv.org/abs/1512.08122.

\bibitem{Sayed2013}
A.~Sayed, ``Diffusion adaptation over networks,'' in \emph{Academic Press
  Library in Signal Processing}, R.~Chellapa and S.~Theodoridis, Eds.\hskip 1em
  plus 0.5em minus 0.4em\relax Elsevier, 2013, vol.~3, pp. 323--454.

\bibitem{Sayed2014}
------, \emph{Adaptation, Learning, and Optimization over Networks}.\hskip 1em
  plus 0.5em minus 0.4em\relax Foundations and {T}rends in {M}achine
  {L}earning, 2014, vol.~7.

\bibitem{Zhu2010}
M.~Zhu and S.~Mart\'{i}nez, ``Discrete-time dynamic average consensus,''
  \emph{Automatica}, vol.~46, pp. 322 -- 329, 2010.

\bibitem{Tatarenko2015}
T.~Tatarenko and B.~Touri, ``Non-convex distributed optimization,'' 2016,
  preprint. Available at http://arxiv.org/abs/1512.00895.

\bibitem{Tsianos2012communication}
K.~Tsianos, S.~Lawlor, and M.~Rabbat, ``Communication/computation tradeoffs in
  consensus-based distributed optimization,'' in \emph{Proc. Advances in Neural
  Information Processing Systems}, Lake Tahoe, USA, Dec. 2012, pp. 1943--1951.

\bibitem{NesterovIntro}
Y.~Nesterov, \emph{Introductory Lectures on Convex Optimization}.\hskip 1em
  plus 0.5em minus 0.4em\relax Springer, 2003.

\end{thebibliography}
